\documentclass[12pt]{article}
\usepackage{hyperref}
\usepackage{amsmath,amscd,amsthm,amsfonts}
\oddsidemargin=0.5cm\topmargin=-1cm \numberwithin{equation}{section}

\theoremstyle{plain}
\newtheorem{theorem}{Theorem}[section]
\newtheorem*{the1}{Theorem A}
\newtheorem*{the2}{Theorem B}
\newtheorem*{the3}{Theorem C}
\newtheorem*{the4}{Theorem D}
\newtheorem*{h-1}{H}
\newtheorem*{h0}{H0}
\newtheorem*{h1}{H1}
\newtheorem*{h2}{H2}
\newtheorem*{h3}{H3}
\newtheorem*{open}{Further Questions}
\newtheorem{definition}{Definition}[section]%[chapter]
\newtheorem{lemma}[theorem]{Lemma}
\newtheorem{coro}{Corollary}[section]

\newtheorem{rem}{Remark}[section]

\begin{document}
\vskip 5.1cm
\title{Entropies of commuting transformations on Hilbert spaces \footnotetext{MR Subject Classification(2010): 37H15, 37A35, 37C85, 37H99.
37D20.\\ Keywords: SRB measures; Pesin's entropy formula; $\mathbb{N}^2$-action; infinite dimensional random dynamical system.}}
\author {Zhiming Li\\
\small {School of Mathematics,}\\
\small {Northwest University, Xi'an, 710127, P.R.China}\\
\small{china-lizhiming@163.com}\\
Yujun Zhu\\
\small {School of Mathematical Sciences}\\
\small {Xiamen University, Xiamen, 361005, P. R. China}\\
\small{yjzhu@xmu.edu.cn}}
\date{}
\maketitle

\begin{center}
\begin{minipage}{130mm}
{\bf Abstract}: By establishing Multiplicative Ergodic Theorem for commutative transformations on a separable infinite dimensional Hilbert space, in this paper, we investigate Pesin's entropy formula and SRB measures of a finitely generated random transformations on such space via its commuting generators. Moreover, as an application, we give a formula of Friedland's entropy for certain $C^{2}$ $\mathbb{N}^2$-actions.
\end{minipage}
\end{center}

\section{\hspace*{-0.15in}. Introduction}

The significance of Pesin's entropy formula (or Ledrappier-Young's entropy formula for SRB measures) lies in its characterizing SRB measures by their Lyapunov exponents and entropy \cite{LY}. Pesin's entropy formula for random transformations and stochastic flows of diffeomorphisms in finite dimensional compact spaces were established in \cite{BL98,LS88,LQ95,K1}. The extension of the above theories to infinite dimensional spaces were presented in \cite{b,l0,l1,l2,L1,L2,T1,T2}. In this paper, we further study the Pesin's entropy formula and SRB measures for random transformations generated by finitely commutative transformations in infinite dimensional Hilbert spaces via its generators, which can be viewed as a generalization of the work in \cite{Hu,Kalinin,z1,z2} to the infinite dimension spaces. However, the techniques and strategies are completely different due to the feature of infinite dimensional smooth dynamics. For more recent progress of SRB measures in infinite dimensional spaces, we refer to the elegant survey \cite{y}.

To obtain the relations of metric entropy of the random transformation and the Lyapunov exponents of its generators, the basic strategy is to estimate the random exponential expanding rate in a deterministic subspace by exponential expanding rates of generators in this subspace. Intuitively, random exponential expanding rate should be the weighted combination of exponential expanding rates of generators, and the weights depend on the probability law of choosing the generators for each iteration. So we first establish Multiplicative Ergodic Theorem (Theorem A) for commutative transformations on a separable infinite dimensional Hilbert space, which is a higher rank group actions version of \cite{LL} and infinite dimensional version of \cite{Hu}. By our assumptions, the deterministic subspace is the common expanding subspace of each generators. Then by comparing the dynamics of the random transformation with the dynamics of its generators, we reformulate Ruelle's entropy inequality (Theorem B), the Pesin's entropy formula and SRB measures (Theorem C) via the generators. Moreover, as an application, we give a formula of Friedland's entropy (Theorem D) for certain $C^{2}$ $\mathbb{N}^2$-actions.

This paper is organized as follows. In Section 2, basic notions such as finitely generated random transformations, Lyapunov exponents, metric entropy and Friedland's entropy will be introduced. Then we will state the main results (Theorem A-Theorem D). Section 3 is devoted to the proofs of the main results.

\section{\hspace*{-0.15in}. Preliminaries and Statement of Main Results}

Let $X$ be a
separable infinite dimensional Hilbert space with inner product
$<\cdot, \cdot>$, norm $\|\cdot\|$, distance function $d$ and
$\sigma$-algebra $\mathcal{B}$ of Borel sets.

\subsection{Deterministic Infinite Dimensional Dynamical Systems}
We begin with the notion of $C^1$ map. Let $L(X, Y)$ denote the
collection of bounded linear operators from Banach space $X$ to $Y$.
Let $U$ be a non-empty open subset of $X$. A measurable map $g:
U\mapsto Y$ is said to be $C^1$ if there exists $\{d_x g:\ X\mapsto
Y\}_{x\in U}$ of $L(X, Y)$ such that i) for each $x\in U$,
\[
\lim\limits_{y\rightarrow x}\frac{\|g(x)-g(y)-d_xg(x-y)\|}{\|x-y\|}=0;
\]
ii) the map $x\to d_x g$ is continuous from $U$ to $L(X, Y)$. The
map $g$ is said to be $C^2$ if its derivative $d_{(\cdot)}g$ is also
$C^1$ from $U$ to $L(X, Y)$.
For any bounded subset $A$ of $X$,  denoted by $\alpha(A)$ the smallest nonnegative
real number $r$ such that $A$ can be covered by finite many Borel
balls of $X$ with radius at most $r$. (It is called the \emph{Kuratowski measure of non-compactness} of the set $A$.)
 Define also the index of
compactness of a map $g: X\to X$ as being the number
\begin{equation}\label{00}
\|g\|_{\alpha}:=\inf\{k>0:\ \ \alpha(g(A))\leq k\alpha(A)\ \mbox{for
any bounded set }A\ \mbox{of}\ X\}.
\end{equation}
In case $g$ is a continuous linear operator, we have
$\|g\|_{\alpha}=\alpha(g(B_{X}))$, where $B_{X}$ is the open unit
ball of $X$. Let $h$ be another continuous linear operator of $X$,
then we have \begin{equation}\label{e-2.1} \|g+h\|_{\alpha}\leq
\|g\|_{\alpha}+\|h\|_{\alpha},\ \ \|g\circ h\|_{\alpha}\leq
\|g\|_{\alpha}\cdot\|h\|_{\alpha}.
\end{equation}
Then for any $C^1$ map $g:X\rightarrow X$ and $g$-invariant compact set $K\subset X$, (\ref{e-2.1}) gives the existence of the limits $$l_{\alpha}(g):=\lim_{n\to\infty}\frac{1}{n} \log \sup_{x\in K}\|d_xg^n\|_{\alpha}$$
and $$l_{\alpha}(x, g):=\lim_{n\to\infty}\frac{1}{n} \log \|d_xg^n\|_\alpha\,\,\mbox{for\,}\mu\,\mbox{almost\,every\,}x\in K,$$where $\mu$ is any $g$-invariant measure.

For $(\xi_{1},\cdots, \xi_{p})\in X^{p}$, $p\in \mathbb{N}$, define
\begin{equation*}
  V_{p}(\xi_1,\cdots, \xi_{p}):=\left(\prod_{i=1}^{p-1}\mbox{dist}(\xi_i, \mbox{span}\{\xi_{i+1},\cdots,
  \xi_{p}\})\right)\cdot \|\xi_{p}\|,
\end{equation*}
where for $i=1,2,\cdots, p-1$, \[\mbox{dist}(\xi_i,
\mbox{span}\{\xi_{i+1},\cdots,
  \xi_{p}\})=\inf\left\{\|\xi_i-\eta\|:\ \ \eta\in \mbox{span}\{\xi_{i+1},\cdots,
  \xi_{p}\}\right\}.\]
For $T\in L(X, X)$, define
\begin{equation}\label{tiji}
V_{p}(T):=\sup_{\|\xi_i\|=1, 1\leq i\leq p}V_{p}(T(\xi_1),\cdots,
T(\xi_{p})).
\end{equation}
%For $(\xi_{1},\cdots, \xi_{p})\in X^{p}$, $p\in \Bbb N$, define
%\begin{equation*}
  %V_{p}(\xi_1,\cdots, \xi_{p}):=\left(\prod_{i=1}^{p-1}\mbox{dist}(\xi_i, \mbox{span}\{\xi_{i+1},\cdots,
  %\xi_{p}\})\right)\cdot \|\xi_{p}\|,
%\end{equation*}
%where for $i=1,2,\cdots, p-1$, \[\mbox{dist}(\xi_i,
%\mbox{span}\{\xi_{i+1},\cdots,
  %\xi_{p}\})=\inf\left\{\|\xi_i-\eta\|:\ \ \eta\in \mbox{span}\{\xi_{i+1},\cdots,
  %\xi_{p}\}\right\}.\]
%For $f\in L(X, X)$, define
%$
%V_{p}(f):=\sup_{\|\xi_i\|=1, 1\leq i\leq p}V_{p}(f(\xi_1),\cdots,
%f(\xi_{p})).$\{V_{p}(d_{x}f(n, \omega))\}_{p\in \Bbb N}$
%$
By a detailed exploration of the asymptotic behaviors of $\{V_{p}(d_{x}g^n)\}_{p\in \Bbb N}$, Lian-Lu \cite{LL} proved
the following theorem concerning the existence of Lyapunov
exponents, we only present the part which is adequate for our purposes.
We need the following assumptions to get Multiplicative Ergodic Theorem.
\begin{h-1}
\ \\
(i) $g$ is $C^{1}$ Fr\'echet differentiable and injective;\\
(ii) the derivative of $g$ at $x \in X$, denoted $d_x g$, is also injective;\\
(iii) there exists a $g$-invariant compact set $K\subset X$.
\end{h-1}
\begin{theorem}\cite[Theorem 3.1]{LL}\label{MET}Suppose $g$ satisfies (H). For any $g$-invariant measure $\mu$, and $\lambda_{\alpha} > l_{\alpha}(x, g)$, there is a measurable, $f$-invariant set
$\Gamma_{g} \subset X$
with $\mu(\Gamma_{g})=1$ and at most finitely many real numbers
$$\lambda_1(x,g) > \lambda_2(x,g) > \cdots > \lambda_{r(x,g)}(x,g)$$
with $\lambda_{r(x,g)}(x,g)> \lambda_{\alpha}$ for which the following properties hold.
For any $x \in \Gamma_{g}$, there is a splitting
$$X = E_1(x,g) \oplus E_2(x,g) \oplus \cdots \oplus E_{r(x,g)}(x,g) \oplus E_{\alpha}(x,g)$$
such that
\begin{itemize}
\item[(a)] for each $j=1,2,\dots, r(x,g)$, $\dim E_j(x,g) = m_j(x,g)$ is finite, $d_x g E_j(x,g) = E_j(g x, g)$, and for any $v \in E_j(x,g) \setminus \{0\}$, we have
$$
\lambda_j(x,g) = \lim_{n \to \infty} \frac{1}{n} \log \|d_x g^n v\| ;
$$
\item[(b)] the distribution $E_{\alpha}(x,g)$ is closed and finite-codimensional, satisfies $d_x g E_{\alpha}(x,g) \subset E_{\alpha}(gx,g)$ and
$$
\lambda_\alpha \geq \limsup_{n \to \infty} \frac{1}{n} \log \|d_x g^n |_{E_{\alpha}(x,g)}\|  \ ;
$$

\item [(c)] for $p\leq \Sigma_{i=1}^{r(x,g)}m_{i}(x,g)$,  we have
$$\lim_{n\rightarrow \infty}\frac{1}{n}\log V_p(d_x g^n)=\sum_{k=1}^{p}\widetilde{\lambda}_{k}(
    x,g),$$
where $\{\widetilde{\lambda}_{k}(x,g)\}$ are $\lambda_j(
    x,g)$'s repeated with multiplicity $m_j(x,g)$;

\item[(d)] the mappings $x \mapsto E_j(x,g), x \mapsto E_{\alpha}(x,g)$ are measurable,
\item[(e)] writing $\pi_j(x,g)$ for the projection of $X$ onto $E_j(x,g)$ via the splitting at $x$,
we have
$$
\lim_{n \to \infty} \frac{1}{n} \log |\pi_j(g^n x,g)| = 0 \quad a. s.
$$
\end{itemize}
\end{theorem}
%Now we give more assumptions on the generators.
%\begin{itemize}
%\item[(H3)] We assume
%$$l_{\alpha}(f_i):=\lim_{n\to\infty}\frac{1}{n} \log \sup_{x\in K}\|d_xf_i^n\|_\alpha \ < \ 0$$

%\vspace{-3 pt}
%\item[(H4)] $E^u(x,f_1)=E^u(x,f_2)$
%for any $x\in K$.

%\vspace{-3 pt}
%\item[(H5)] $f_i : K\mapsto K$ is a homeomorphism.
%\vspace{-3 pt}
%\item[(H5)] $(f_i,\mu_i)$ has no zero Lyapunov exponents.
%\vspace{-3 pt}
%\item[(H6)] The set $K$ in (H2) has finite box-counting dimension.
%\end{itemize}

In order to get SRB measures, it is necessary to put some restrictions
on $g$. Under above setting, it is true by
seeing Theorem \ref{MET} that $l_{\alpha}(g)<0$ implies the
existence of Lyapunov exponents $\lambda_1(x, g)>\lambda_2(x, g)>\cdots$ with multiplicities $m_1(x, g
), m_2(x, g),\cdots$ at $\mu$-a.e. $x$, which can be
infinitely many but only admits finitely many positive ones.

\begin{theorem}\cite[Theorem 1.1]{l1} Suppose $g$ is $C^{2}$ Fr\'echet differentiable, satisfies (H) and $l_{\alpha}(g)<0$. $\mu$ is supported on $K$ with $h_{\mu}(g)<\infty$, where $K$ is a compact invariant set. If $\mu$ is an SRB measure, then
\begin{equation*}
h_{\mu}(g) = \int \sum_{\lambda_j(x, g)>0} m_j(x, g) \lambda_j(x, g)d\mu.
\end{equation*}
The converse is also true if $(g,\mu)$ has no zero Lyapunov exponents and the set $K$ has finite box-counting dimension.
\end{theorem}
\subsection{Random Transformations with Finite Commuting Generators}
\subsubsection{Multiplicative Ergodic Theorem for $\mathbb{N}^2$-actions}
We now consider $\mathbb{N}^N$ actions generated by commuting maps $\mathfrak{F}=\{f_i: X \to X\}_{i=1,\ldots,N}$, in which $f_i\circ f_j=f_j\circ f_i$ for all $1\leq i,j\leq N$. For
simplicity of the notations, we assume $N=2$. We denote by $\mathcal{M}_i$ (resp. $\mathcal{M}_i^e$) the set of (resp. ergodic) Borel probability measures on $X$ which are invariant under $f_i$, for any $i=1,2$. Let $\mathcal{M}=\mathcal{M}_1\cap\mathcal{M}_2$ and $\mathcal{M}^e=\mathcal{M}_1^e\cap\mathcal{M}_2^e$. By \cite[Proposition 1.3, 1.4]{Hu}, $\mathcal{M}\neq\emptyset$ and $\mathcal{M}^e\neq\emptyset$. Our first result of Multiplicative Ergodic Theorem for infinite dimensional $\mathbb{N}^2$-actions will prove extremely useful in the next section.
We need the following assumptions on generators which will be needed throughout the paper.
\begin{h0}
\ \\
(i) $f_1$, $f_2$ are $C^{1}$ Fr\'echet differentiable and injective;\\
(ii) the derivatives of $f_1$ and $f_2$ are also injective;\\
(iii) there exists a compact set $K\subset X$ such that $f_1(K)=K$ and $f_2(K)=K$.
\end{h0}
\begin{the1}\label{thm1}
Suppose $f_1$, $f_2$ satisfy (H0) and $\mu\in \mathcal{M}$ with supp$\mu\subset K$. For any $\lambda_{\alpha} >\max\{ l_{\alpha}(f_1),l_{\alpha}(f_2)\}$, there exists a measurable set $\Gamma\subset \Gamma_{f_i}$ with $f_i(\Gamma)=\Gamma$ for each $i=1,2$ and $\mu(\Gamma)=1$, such that for any $x\in \Gamma$, there exists %at most
%finitely many real numbers $\lambda_j(f_i,x)$ with $\lambda_{r(f_i,x)}(f_i,x)> \lambda_{\alpha}$
a decomposition of $X$ into at most finitely many subspaces
$$
X=\bigoplus_{j_1=1}^{r(x,f_1)}
\bigoplus_{j_2=1}^{r(x,f_2)}E_{j_1,j_2}(x)
\bigoplus_{j_2=1}^{r(x,f_2)}E_{r(x,f_1)+1,j_2}(x)
\bigoplus_{j_1=1}^{r(x,f_1)}E_{j_1,r(x,f_2)+1}\bigoplus E_{\alpha}(x)
$$
satisfying the following properties:
\begin{itemize}

\item[(a)] if $E_{j_1,j_2}(x)\neq\{0\}$, for any $s_1,s_2\in \mathbb{Z}^+$, $0\neq v\in E_{j_1,j_2}(x)$, $1\leq j_i \leq r(x,f_i)$ and $i=1,2$,
\begin{equation}\label{0}
\lim_{n\longrightarrow \infty}\frac{1}{n}\log\|d_x (f_1^{s_1}\circ f_2^{s_2})^nv\|=s_1\lambda_{j_1}(x,f_1)+s_2\lambda_{j_2}(x,f_2);
\end{equation}

\item[(b)] if $E_{j_1,r(x,f_2)+1}(x)\neq\{0\}$, for $0\neq v\in E_{j_1,r(x,f_2)+1}(x)$, $1\leq j_1 \leq r(x,f_1)$,
$$
\lim_{n\longrightarrow \infty}\frac{1}{n}\log\|d_x f_1^nv\|=\lambda_{j_1}(x,f_1)
$$
and
$$
\limsup_{n\longrightarrow \infty}\frac{1}{n}\log\|d_x f_2^nv\|\leq\lambda_{\alpha};
$$

\item[(c)] if $E_{r(x,f_2)+1, j_2}(x)\neq\{0\}$, for $0\neq v\in E_{r(x,f_2)+1, j_2}(x)$, $1\leq j_2 \leq r(x,f_2)$,
$$
\lim_{n\longrightarrow \infty}\frac{1}{n}\log\|d_x f_2^nv\|=\lambda_{j_2}(x,f_2)
$$
and
$$
\limsup_{n\longrightarrow \infty}\frac{1}{n}\log\|d_x f_1^nv\|\leq\lambda_{\alpha};
$$

\item[(d)] for $0\neq v\in E_{\alpha}(x)$ and $i=1,2$,
$$
\lim_{n\longrightarrow \infty}\frac{1}{n}\log\|d_x f_i^nv\|\leq\lambda_{\alpha};
$$

\item[(e)]for $1\leq j_i \leq r(x,f_i)$ and $i=1,2$, each $E_{j_1,j_2}(x)$, we have the following invariance properties:

$df_i(x)E_{j_1,j_2}(x)=E_{j_1,j_2}(f_i(x))$ and $\lambda_{j_i}(f_{i'}(x),f_i)=\lambda_{j_i}(x,f_i)$,
where $1\le i,i'\le 2$;

\item[(f)]for $1\leq j_i \leq r(x,f_i)$ and $i=1,2$, each $E_{r(x,f_1)+1,j_2}(x)$ and $E_{j_1,r(x,f_2)+1}(x)$ we have the following invariance properties:

$df_1(x)E_{r(x,f_1)+1,j_2}(x)\subset E_{r(x,f_1)+1,j_2}(f_1(x))$, $df_2(x)E_{j_1,r(x,f_2)+1}(x)\subset E_{j_1,r(x,f_2)+1}(f_2(x))$  $df_2(x)E_{r(x,f_1)+1,j_2}(x)= E_{r(x,f_1)+1,j_2}(f_2(x))$, $df_1(x)E_{j_1,r(x,f_2)+1}(x)= E_{j_1,r(x,f_2)+1}(f_1(x))$;

\item[(g)]writing $\pi_{j_1,j_2}(x)$ for the projection of $X$ onto $E_{j_1,j_2}(x)$ via the splitting at $x$, for every $i=1,2$,
we have
\begin{equation}\label{01}
\lim_{n \to \infty} \frac{1}{n} \log |\pi_{j_1,j_2}(f_i^n x)| = 0 \quad a. s.
\end{equation}
\end{itemize}
\end{the1}

\subsubsection{Random Transformations with Finite Commuting Generators}
Again, without loss of generality, we consider random transformations generated by two maps $\mathfrak{F}=\{f_1,f_2\}$ with $f_1\circ f_2=f_2\circ f_1$.

%\begin{itemize}
%\item[(H0)] $f_1\circ f_2=f_2\circ f_1$.

%\vspace{-3 pt}
%\item[(H1)] (i) $f_i$ is $C^{2}$ Fr\'echet differentiable and injective;

%\vspace{-3 pt}
%(ii) the derivative of $f_i$ at $x \in X$, denoted $d_x f_i$, is also injective.

%\vspace{-3 pt}
%\item[(H2)] (i) $f_i$ leaves invariant a compact set $K \subset X$, with $f_i(K)=K$.
%\end{itemize}

Let $\Omega=\mathfrak{F}^{\mathbb{N}}=\prod_{0}^\infty \mathfrak{F}$
be the infinite product of $\mathfrak{F}$, endowed with the product topology and the product Borel $\sigma$-algebra $\mathcal{A}$,
and let $\theta$ be the left shift operator on $\Omega$ which is defined by $(\theta\omega)_n=\omega_{n+1}$ for $\omega=(\omega_n)\in \Omega$. Given $\omega=(\omega_n)\in \Omega$,
we write $f_{\omega}=\omega_0$ and
$$
f(n, \omega):=\left\{\begin{array}{ll}
f_{\theta^{n-1}\omega}\circ\cdots\circ f_{\theta\omega}\circ f_{\omega} \quad &\;\;n>0\\
id \quad &\;\;n=0.
\end{array}\right.
$$
There is a natural skew product transformation $F: \Omega\times X\longrightarrow \Omega\times X$ over $(\Omega, \theta)$ which is defined by $F(\omega, x)=(\theta\omega, f_{\omega}(x)).$
For any probability measure $\nu$ on $\mathfrak{F}$, we can define a probability
measure $\mathbf{P}_{\nu}=\nu^{\mathbb{N}}$ on $\Omega$ which is invariant with respect
to $\theta$. By the induced \emph{finitely generated (i.i.d.) random transformation $f$ over $(\Omega, \mathcal{A}, \mathbf{P}_{\nu}, \theta)$} we  mean the system generated by the randomly composition of $f_{i}$, $i=1,2$ in the law of $\nu$.  We are interested in dynamical behaviors of these actions
for $\mathbf{P}_{\nu}$-a.e. $\omega$ or on the average on $\omega$.
%%Let $X$ be a
%separable infinite dimensional Hilbert space with inner product
%$<\cdot, \cdot>$, norm $\|\cdot\|$, distance function $d$ and
%$\sigma$-algebra $\mathcal{B}$ of Borel sets. Let $(\Omega,\mathcal{F})$ be a countably generated probability space with a measurable flow $(\theta^t)_{t\in\mathbb{R}}$ on it preserving an ergodic probability measure $\mathbf{P}$.
%\begin{definition}
%A map $$\phi:\mathbb{R}^+\times \Omega\times X\longrightarrow X$$
%$$(t,\omega,x)\mapsto \phi(t,\omega)x$$ is called a (measurable) \emph{stochastic flow} on $X$ over $(\Omega,
%\mathcal{F}, \mathbf{P}, \theta^t)$,if \par(1) $\phi$ is
%$(\mathcal{B}(\mathbb{R}^+)
%\otimes\mathcal{F}\otimes\mathcal{B}(X),\mathcal{B}(X))$-measurable;
%\par (2) The maps $\phi(t,\omega):X \rightarrow X$ form a cocycle
%over $\theta$,i.e., they satisfy $\phi(0,\omega)=id$,
%$\phi(t+s,\omega)=\phi(t,\theta^{s}\omega)\circ\phi(s,\omega)$ for
%all $s, t\in \mathbb{R}^+,\omega\in\Omega.$\par
%\end{definition}
%Moreover, for $t\in \mathbb{R}^+$, we call
It is clear that $f(n,\omega)$ is injective and strongly measurable for any $n\in\mathbb{N}$, (in the sense for each $x\in X$
fixed, the map $\omega\mapsto f(n,\omega)(x)$ is measurable from $\Omega$ to $X$).

A Borel probability measure $\mu$ on $X$ is called \emph{$f$-invariant} if
$
\int_{\Omega} \mu(f_{\omega}^{-1}A)d\mathbf{P}_{\nu}(\omega)=\mu(A)
$
for all Borel $A\subset X$. We denote by $\mathcal{M}_f$ (resp. $\mathcal{M}_f^e$) the set of all $f$-invariant (resp. ergodic) Borel probability measures. Clearly, $\mathcal{M}\subset\mathcal{M}_f$ and $\mathcal{M}^e\subset\mathcal{M}_f^e$.

%\subsection{\hspace*{-0.15in}. Metric entropy and Lyapunov exponents}\label{lyapunov1}
%\begin{definition}For a finite or countable Borel partition $\mathcal{P}$ of $\Omega\times X$, $\mu\in\mathcal{M}_f$, the limit
%\begin{equation*}
%h_{\mu}(f,\mathcal{P}):=\lim_{n\rightarrow\infty}\frac{1}{n}\int_{\Omega}
%H_{\mu}
%\big(\bigvee_{i=0}^{n-1}(f(\omega,i)^{-1}\mathcal{P^{\omega}}
%\big)\;d\mathbf{P}_{\nu}(\omega),
%\end{equation*}
%where $H_{\mu}(\mathcal{Q}):=-\sum\limits_{A\in\mathcal{Q}}\mu(A)\log\mu(A)$ for a finite or countable
%partition $\mathcal{Q}$ of $X$, exists. The number
%$$
%h_{\mu}(f):=\sup_{\mathcal{P}}h_{\mu}(f,\mathcal{P}),
%$$
%where $\mathcal{P}$ ranges over all finite  or countable partitions of $\Omega\times X$, is called the
%\emph{measure-theoretic entropy} of $f$.
%\end{definition}

For each $\omega\in \Omega$,
using (\ref{e-2.1}), we see that for  $x\in
K$, $m, n\in \mathbb{N}$,
\begin{equation*}\log \|d_x f(
n+m,\omega)\|_{\alpha}\leq \log \|d_{x}f(n,\omega)\|_{\alpha}+\log \|d_{f(n,
\omega)x}f(n,\theta^n\omega)\|_{\alpha}.\end{equation*} This  gives the
existence of the limits
$$
l_{\alpha}(\omega, x):=\lim\limits_{n\rightarrow
\infty}\frac{1}{n}\log \|d_x f(n, \omega)\|_{\alpha}
$$ and
$$
l_{\alpha}(\omega, f):=\lim_{n\to\infty}\frac{1}{n} \log \sup_{x\in K}\|d_xf(n,\omega)\|_\alpha,
 $$where $K\subset X$ is a compact set such that $f_1(K)=K$ and $f_2(K)=K$.
 By Random Subadditive Ergodic Theorem \cite[Theorem 2.2]{K1}, one can show that $l_{\alpha}(\omega, x)$ and $l_{\alpha}(\omega, f)$ are non-random in the sense that there is a measurable, $f$-invariant set
$\Gamma_0\subset X$
with $\mu(\Gamma_{0})=1$ such that for $\mathbf{P}_{\nu}$-a.e $\omega\in \Omega$, $l_{\alpha}(\omega, x)=l_{\alpha}(x)$ and $l_{\alpha}(\omega, f)=l_{\alpha}(f)$ for any $x\in \Gamma_0$.
For the Lyapunov exponents for random transformations, we only present the part which is adequate for our purposes.

\begin{theorem}\label{thm2.1}\cite[Theorem 3.1]{LL}
  Let $f$ be a finitely generated random transformation on an infinite dimensional Banach
  space $X$ over $(\Omega, \mathcal{A}, \mathbf{P}_{\nu}, \theta)$ by $f_1$ and $f_2$. Suppose $\mu\in \mathcal{M}$ and $f_1$, $f_2$ satisfy (H0).
   For any $\lambda_{\alpha} > l_{\alpha}(f)$, there is a measurable, $f$-invariant set
$\Gamma_f\subset X$
with $\mu(\Gamma_{f})=1$. For any $x\in \Gamma_{f}$, there are at most finitely many real numbers
$$\lambda_1(x,f) > \lambda_2(x,f) > \cdots > \lambda_{r(x,f)}(x,f)$$
with $\lambda_{r(x,f)}(x,f)> \lambda_{\alpha}$ for which the following properties hold.
For $\mathbf{P}_{\nu}$-a.e $\omega\in \Omega$, there is a splitting
$$X = E_1(\omega, x,f) \oplus E_2(\omega, x,f) \oplus \cdots \oplus E_{r(x,f)}(\omega, x,f) \oplus E_{\alpha}(\omega, x,f)$$
such that
\begin{itemize}
\item[(a)] for each $j=1,2,\dots, r(x,f)$, $\dim E_j(\omega, x,f) = m_j(x,f)$ is finite, $d_x f_{\omega} E_j(\omega, x,f) = E_j(F(\omega,x), f)$, and for any $v \in E_j(\omega, x,f) \setminus \{0\}$, we have
$$
\lambda_j(x,f) = \lim_{n \to \infty} \frac{1}{n} \log \|d_x f(n,\omega)v\|;
$$
\item[(b)] the distribution $E_{\alpha}$ is closed and finite-codimensional, satisfies $d_x f_{\omega} E_{\alpha}(\omega, x,f) \subset E_{\alpha}(F(\omega,x), f)$ and
$$
\lambda_\alpha \geq \limsup_{n \to \infty} \frac{1}{n} \log \|d_x f(n,\omega) |_{E_{\alpha}(\omega, x,f)}\|  \ ;
$$
\item [(c)] for $p\leq \Sigma_{j=1}^{r(x,f)}m_j(x,f)$,  we have
$$\lim_{n\rightarrow \infty}\frac{1}{n} \log V_p(d_x f(n, \omega))=\sum_{k=1}^{p}\widetilde{\lambda}_{k}(
    x,f),$$
 where $\{\widetilde{\lambda}_k(x,f)\}$ are $\lambda_j(
    x,f)$'s repeated with multiplicity $m_j(x,f)$;

\item[(d)] the mappings $(\omega,x) \mapsto E_j(\omega, x,f), (\omega,x)\mapsto E_{\alpha}(\omega, x,f)$ are measurable,
\item[(e)] writing $\pi_j(\omega, x,f)$ for the projection of $X$ onto $E_j(\omega, x,f)$ via the splitting at $x$,
we have
$$
\lim_{n \to \infty} \frac{1}{n} \log |\pi_j(F^n(\omega,x), f)| = 0 \quad a. s.
$$
\end{itemize}
\end{theorem}

By \cite[Theorem 2.7]{l1}, for $\mathbf{P}_{\nu}\times\mu$-a.e. $(\omega, x)$ the unstable set
\[
W^u(\omega,x):=\{y\in X:\ \limsup_{n\rightarrow +\infty}\frac{1}{n}\log
d(f(-n,\omega)x, f(-n,\omega)y)<0\}
\]
is a $C^{1,1}$ immersed Hilbert manifold of $X$, the so called
\emph{unstable manifold at $(\omega, x)$}. A measurable partition $\eta$
of $\Omega\times X$ is  \emph{subordinate to $W^u$
manifolds} of $(f, \mu)$, if for $\mathbf{P}_{\nu}\times\mu$-a.e. $(\omega,
x)$, denote by $\eta(\omega,x)$ the element of $\eta$ that contains $(\omega,
x)$, then
\[
\eta^{\omega}(x):=\{y:\ (w, y)\in \eta(\omega, x)\}\subset W^u(\omega, x)
\]
and $\eta^{\omega}(x)$ contains an open neighborhood of $x$ in $W^u(\omega,x)$, this
neighborhood being taken in the submanifold topology of $W^u(\omega,x)$.
A Borel probability measure $\mu$ is said to have
\emph{absolutely continuous conditional measures on $W^u$-manifolds}
of $(f, \mu)$, if for any measurable partition
$\eta$ subordinate to $W^u$-manifolds of the system, one has
$\mu^{\eta^{\omega}(x)}\ll \mbox{Leb}_{(\omega,x)}^{u}$, $\mathbf{P}_{\nu}\times\mu$-a.e., where
$\{\mu^{\eta^{\omega}(x)}\}_{x\in K}$ is a canonical system of the
conditional measures of $\mu$ associated with the partition
$\{\eta^{\omega}(x)\}_{x\in K}$ of $X$ and $\mbox{Leb}_{(\omega,x)}^{u}$ is
the Lebesgue measure on $W^u(\omega, x)$ induced by its inherited
Riemannian metric as a submanifold of $X$. We call such measure an \emph{SRB measure}. Similarly, we denote by $W^u(x, f_i)$ the unstable manifold of $(f_i, \mu_i)$, $i=1,2$.

Now we give more assumptions on the generators.
\begin{h1}
\ \\ $l_{\alpha}(f_1)<0$ and $l_{\alpha}(f_2)<0$.

%\vspace{-3 pt}
%\item[(H2)]$E^u(x,f_1)=E^u(x,f_2)$
%for any $x\in K$.

%\item[(H5)] $(f_i,\mu_i)$ has no zero Lyapunov exponents.
%\vspace{-3 pt}
%\item[(H6)] The set $K$ in (H2) has finite box-counting dimension.
\end{h1}

%Under the above setting, it is true by
%seeing Theorem \ref{MET} that (H1-H3) implies the
%existence of Lyapunov exponents $\lambda_1(x, f_i)>\lambda_2(x, f_i)>\cdots$ with multiplicities $m_1(x, f_i
%)>m_2(x, f_i)>\cdots$ at $\mu_i$-a.e. $x$, which can be
%infinitely many but only admits finitely many positive ones for $i=1,2$.

%\begin{theorem}\cite[Theorem 1.1]{l1} For any $i=1,2$, suppose $(f_i,\mu_i)$ satisfies (H1-H3) above and $h_{\mu_i}(f_i)<\infty$. If $\mu_i$ is an SRB measure, then
%\begin{equation*}
%h_{\mu_i}(f_i) = \int \sum_{\lambda_j(x, f_i)>0} m_j(x, f_i) \lambda_j^+(x, f_i)d\mu_i.
%\end{equation*}
%The converse is also true if (H4-H6) hold.
%\end{theorem}

%Similarly, there are Lyapunov exponents $\lambda_1(x,
%f)>\lambda_2(x,f)>\cdots$ with multiplicities $m_1(x,
%f)>m_2(x,f)>\cdots$ at $\mu$-a.e. $x$, for the induced random transformation $f$ over $(\Omega, \mathcal{A}, \mathbf{P}_{\nu}, \theta)$. By Lemma \ref{t1}, there are only finitely many positive ones, and the following relation between Pesin's entropy formula and SRB property for $f$ holds.

%\begin{theorem}\cite[Theorem 1.1]{l1} Let $f$ be a finitely generated random transformation on an infinite dimensional Hilbert
%space $X$ over $(\Omega, \mathcal{A}, \mathbf{P}_{\nu}, \theta)$ by $f_1$ and $f_2$. Suppose $\mu\in \mathcal{M}$ and $(f_i,\mu)$ satisfies (H1-H3) above and $h_\mu(f)<\infty$. If $\mu$ is an SRB measure, then
%\begin{equation*}
%h_\mu(f) = \int \sum_{\lambda_j(x,f)>0} m_j(f, x) \lambda_j^+(x,f)d\mu.
%\end{equation*}
%The converse is also true if (H4-H6) hold.
%\end{theorem}
We are in a
situation to state  the main results of this paper.
\begin{the2}\label{thm2}
 Let $f$ be a finitely generated random transformation of an infinite dimensional Banach
  space $X$ over $(\Omega, \mathcal{A}, \mathbf{P}_{\nu}, \theta)$. Suppose $\mu\in \mathcal{M}$ and $(f_i,\mu)$ satisfies (H0-H1) above, then Ruelle's inequality
\begin{equation}\label{m1}
h_{\mu}(f)\leq \int
\sum_{i=1}^2\sum_{\lambda_{j_k}(x,f_i)>0}\nu(f_i)\lambda_{j_k}(x,f_i)m_{j_k}(x,f_i)\ \ d\mu
\end{equation}
holds, where $h_{\mu}(f)$ is the metric entropy of $f$.
\end{the2}
For Pesin's entropy formula, we need more smooth condition of the maps, thus we replace (H0) with the following conditions:
\begin{h2}
\ \\
(i) $f_1$, $f_2$ are $C^{2}$ Fr\'echet differentiable and injective;\\
(ii) the derivatives of $f_1$ and $f_2$ are also injective;\\
(iii) there exists a compact set $K\subset X$ such that $f_1(K)=K$ and $f_2(K)=K$.
\end{h2}
\begin{the3}\label{thm3}
  Let $f$ be a finitely generated random transformation of an infinite dimensional Hilbert
  space $X$ over $(\Omega, \mathcal{A}, \mathbf{P}_{\nu}, \theta)$. Suppose $\mu\in \mathcal{M}$ with supp$\mu\subset K$, $(f_i,\mu)$ satisfies (H1-H2) and $h_{\mu}(f)<+\infty$. If $\mu$ is an SRB measure, then $$h_{\mu}(f)\geq\int \sum_{i=1}^2\sum_{\lambda_{j_k}(x,f_i)>0}\nu(f_i)\lambda_{j_k}(x,f_i)d_{j_k}(x,f_i)
d\mu,$$ where $d_{j_1}(x,f_1)=m_{j_1}(x,f_1)-dim E_{j_1,r(x,f_2)+1}(x)$, $d_{j_2}(x,f_2)=m_{j_2}(x,f_2)-dim E_{r(x,f_1)+1,j_2}(x)$.
\end{the3}
If the following assumption (H3) on the generators are made, we will get Pesin's entropy formula and look more closely at SRB measures. The main purpose in making such assumption lies in the fact that we lose control of the random transformation when the stable and unstable directions of the generators mixes together with an infinite dimensional freedom. A trivial motivative example is the random transformations generated by hyperbolic torus automorphisms $f_1=\begin{bmatrix}
   2 & 1 \\
   1 & 1
  \end{bmatrix} $ and $f_2=f_1^{-1}$ with $\nu(f_1)=\nu(f_2)=\frac{1}{2}.$ It is easy to see that Corollary 2.1 and Corollary 2.2 fail without  (H3).

More precisely, let
$\lambda_{\alpha}=0$ in Theorem A and denote by
$$
 E^u(x, f_1):=\oplus_{j=1}^{r(x,f_1)} E_j(x, f_1),\;\;E^u(x, f_2):=\oplus_{j=1}^{r(x,f_2)} E_j(x, f_2).
$$
\begin{h3}
\ \\
$E^u(x,f_1)=E^u(x,f_2)$
for any $x\in K$.
\end{h3}
\begin{coro}\label{c1}
Let $f$ be a finitely generated random transformation of an infinite dimensional Hilbert
  space $X$ over $(\Omega, \mathcal{A}, \mathbf{P}_{\nu}, \theta)$. Suppose $\mu\in \mathcal{M}$ with supp$\mu\subset K$ and $(f_i,\mu)$ satisfies (H1-H3) and $h_{\mu}(f)<+\infty$. Then Pesin's entropy formula
  \begin{equation}\label{m2}
h_{\mu}(f)=\int \sum_{i=1}^2\sum_{\lambda_{j_k}(x,f_i)>0}\nu(f_i)\lambda_{j_k}(x,f_i)m_{j_k}(x,f_i)
d\mu  \end{equation}
  holds if $\mu$ is an SRB measure.
\end{coro}
\begin{coro}\label{c2}
Let $f$ be a finitely generated $C^2$ random transformation of an infinite dimensional Hilbert
  space $X$ over $(\Omega, \mathcal{A}, \mathbf{P}_{\nu}, \theta)$. Suppose $\mu\in \mathcal{M}$ with supp$\mu\subset K$ and $(f_i,\mu)$ satisfies (H1-H3) above and $h_{\mu}(f)<+\infty$. Then
\begin{itemize}
\item [(a)] $h_{\mu}(f)\geq\sum_{i=1}^{2}\nu_i h_{\mu}(f_i)$ if $\mu$ is an SRB measure of $f$;
\item [(b)] $h_{\mu}(f)\leq\sum_{i=1}^{2}\nu_i h_{\mu}(f_i)$ if $\mu$ is an SRB measure of $f_1$ and $f_2$;
\item [(c)] $h_{\mu}(f)=\sum_{i=1}^{2}\nu_i h_{\mu}(f_i)$ if $\mu$ is an SRB measure of $f$, $f_1$ and $f_2$.
\end{itemize}
\end{coro}

To date, to the best of our knowledge, there has been little discussion
of relation of SRB measures of finitely generated smooth random transformation and the SRB measures of its generators. This paper only severs as a first
attempt towards this direction, and the results are still far from satisfaction. The assumption (H3) in this setting seems artificial and redundant, but we can not remove it for technical reasons. We believe that if the generators have common SRB measures, then they could be SRB measures of the random transformation, and if we add some mild conditions (for example condition H3) on the generators the converse could hold true. We leave them as further questions.
\begin{open}
\begin{itemize}

%\item [(a)] Are Corollary \ref{c1} and Corollary \ref{c2} true without of assumption (H3)?
\item [(a)] Does equality (\ref{m2}) imply that $\mu$ is an SRB measure by adding assumption that $(f_i,\mu)$ has no zero Lyapunov exponents and the set $K$ has finite box-counting dimension?
\item [(b)] If $\mu$ is an SRB measure of every generators, then is $\mu$ an SRB measure of $f$?
\item [(c)] If $\mu$ is an SRB measure of $f$, and assumption (H3) is satisfied, then is $\mu$ an SRB measure of every generators?
\end{itemize}
\end{open}
\subsection{\hspace*{-0.15in}. Friedland's entropy of  $\mathbb{N}^2$-actions}
Friedland's entropy of  $\mathbb{N}^k$-actions was introduced by Friedland \cite{Friedland} via the topological entropy of the shift map on the induced orbit space. More precisely,
let  $\mathbf{f}:\mathbb{N}^2\longrightarrow C^r(K, K) (r\ge 0)$ be a $\mathbb{N}^2$-action on $X$ with the generators $\{f_i\}_{i=1}^2$. Define the \emph{orbit space} of $\mathbf{f}$ by
$$
K_{\mathbf{f}}=\bigl\{\bar{x}=\{x_n\}_{n\in
{\mathbb{N}}}\in\prod_{n\in{\mathbb{Z}}}K : \text{ for any }  n\in{\mathbb{N}}, f_{i_n}(x_n)=x_{n+1}
\text{ for some } f_{i_n}\in\{f_i\}_{i=1}^2 \bigr\}.
$$
This is a closed subset of the compact space
$\prod\limits_{n\in{\mathbb{Z}}}K$ and so is again compact. A
natural metric $\bar{d}$ on $K_{\mathbf{f}}$ is defined by
\begin{equation}\label{metric}
\bar{d}\bigl(\bar{x},\bar{y}\bigr)=\sum_{n=0}^\infty\frac{d(x_n,y_n)}{2^{n}}
\end{equation}
for $\bar{x}=\{x_n\}_{n\in
{\mathbb{N}}}, \bar{y}=\{y_n\}_{n\in
{\mathbb{N}}}\in K_{\mathbf{f}}$.
We can define a shift map
$$
\sigma_{\mathbf{f}}: K_{\mathbf{f}}\rightarrow K_{\mathbf{f}},\;\sigma_{\mathbf{f}}(\{x_n\}_{n\in {\mathbb{N}}})=\{x_{n+1}\}_{n\in
{\mathbb{N}}}.
$$
 Thus we have associated an
${\mathbb{N}}$-action $\sigma_{\mathbf{f}}$ with the ${\mathbb{N}}^2$-action $\mathbf{f}$.

\begin{definition}
\emph{Friedland' s entropy} of an
${\mathbb{N}}^2$-action $\mathbf{f}$ is defined by the topological entropy of the shift
map $\sigma_{\mathbf{f}}: K_{\mathbf{f}}\rightarrow K_{\mathbf{f}}$, i.e.,
\begin{equation}\label{Friedland}
h(\sigma_{\mathbf{f}})=\lim_{\varepsilon\rightarrow
0}\limsup_{n\rightarrow\infty}\frac{1}{n}\log
\;s_{\bar{d}}(\sigma_{\mathbf{f}},n,\varepsilon, X_{\mathbf{f}}),
\end{equation}
where $s_{\bar{d}}(\sigma_{\mathbf{f}},n,\varepsilon, K_{\mathbf{f}})$ is the largest cardinality of any $(\sigma_{\mathbf{f}},n, \varepsilon)$-separated sets of $K_{\mathbf{f}}$.
\end{definition}
Unlike the classical entropy  for $\mathbb{N}^2$ -actions, Friedland's entropy is positive when the generators have finite entropy as single transformations. From the known results about Friedland's entropy, we can see that it is not an easy task to compute it, even for some ``simple" examples (see for example \cite{Friedland,Geller,Einsiedler}).

 However, applying the entropy formula (\ref{m2}) for finitely generated random transformation, we give some formulas and bounds of Friedland's entropy for smooth $\mathbb{N}^2$-actions in a infinite dimensional Hilbert space. %Hence, the Friedland's entropy formulas in \cite{Geller,Einsiedler} are all special cases of the following  result.

 \begin{the4}\label{444}
Let $\mathbf{f}:\mathbb{N}^2\longrightarrow C^2(X, X)$ be a $C^2$ $\mathbb{N}^2$-action on an infinite dimensional Hilbert
  space $X$. Suppose $\mu\in \mathcal{M}_f$ with supp$\mu\subset K$ and $(f_i,\mu)$ satisfies (H1-H3) above and $h_{\mu}(f)<+\infty$, where $f$ is a random transformation generated by $\{f_1,f_2\}$ over $(\Omega, \mathcal{A}, \mathbf{P}_{\nu}, \theta)$. If $\mathbf{P}_{\nu}\times\mu$ is a measure with maximal entropy of $F$ and $\mu$ is an SRB measure, then
\begin{equation}\label{Fried1}
h(\sigma_{\mathbf{f}})\leq -\sum_{i=1}^2\nu(f_i)\log\nu(f_i) +\int \sum_{i=1}^2\sum_{\lambda_{j_k}(x,f_i)>0}\nu(f_i)\lambda_{j_k}(x,f_i)m_{j_k}(x,f_i)
d\mu.
\end{equation}
 Furthermore, if  $\mu\in\mathcal{M}_f^e$ and $\mu (\{x\in X:f_1(x)=f_2(x)\})=0$, then we get the following formula of Friedland's entropy
\begin{equation}\label{Fried3}
h(\sigma_{\mathbf{f}})=\log
\Big(\sum_{i=1}^2\exp(\sum_{\lambda_{j_k}(x,f_i)>0}\lambda_{j_k}(x,f_i)m_{j_k}(x,f_i))\Big).
\end{equation}
\end{the4}
\begin{rem}
In Theorem D, we require that the invariant measure of $F$ is in the form of $\mathbf{P}_{\nu}\times\mu$, we can see \cite[section 3.4 ]{Liu01} for the existence of such a measure for certain systems.
\end{rem}
%A  random measurable set $U\in \mathcal{F}\otimes \mathcal{B}$ is called open if its $\omega$-sections $U_{\omega}:=\{x\in X:\ (\omega,
%x)\in U\}, \omega\in \Omega$ are open for any $\omega\in\Omega$. Given an invariant random compact set $K$
%of $\phi$. We say that $\phi$ is $C^1$  over $K$ if
%there exists some random invariant open set $U$ such that for
%$\mathbf{P}$-a.e. $w$  we have $K_{\omega}\subset U_{\omega}$,  $\phi(t,\omega):\ U_{w}\to X$
%is $C^1$, and  for each $(x, y, v)\in U_{\omega}\times U_{\omega}\times X$, the map $\omega\mapsto
%(\phi(t,\omega)x, D_{y}\phi(t,\omega)(v))$ is measurable from $\Omega$ to $X\times X$.

%In our setting, we have some remarks on Theorem \ref{thm2.1}.  First we notice that for $\mu$-a.e. $(\omega,x)$, there exists only finitely many positive Lyapunov exponents. This is a direct consequence of the theorem and the following lemma

% \begin{lemma}\label{lem2.2}
%  For $\mathbf{P}$-a.e. $\omega$,
%  \[
%\lim\limits_{t\rightarrow +\infty}\frac{1}{t}\log \sup_{x\in
%K_{\omega}}\|D_x \phi(t,\omega)\|_{\alpha}=\lim\limits_{t\rightarrow
%\infty}\frac{1}{t}\int \log \sup_{x\in K_{\omega}}\|D_x \phi(t,\omega)\|_{\alpha}\
%d\mathbf{P}(\omega)<0.
%  \]\lim\limits
%\end{lemma}

\section{\hspace*{-0.15in}. Proofs of Theorem A-Theorem D}
\subsection{Proof of Theorem A}
Recall that $\Gamma_{f_i}\subset X$ is a full measure set such that $f_i\Gamma_{f_i}=\Gamma_{f_i}$, for any $x\in\Gamma_{f_i}$,
$\lambda(x,v,f_i)=\limsup_{n\rightarrow\infty}\frac{1}{n}\log\|d_x f_i^n v\|$, for any $v\in X.$

\begin{lemma}\label{1}
For all
$i,j=1,2$, $i\neq j$, we have
$$
\lambda(f_ix,d_x f_iv,f_j)=\lambda(x,v,f_j)\;\;\text{and}\;\; l_{\alpha}(x, f_i)=l_{\alpha}(f_j x, f_i).
$$
\end{lemma}

\begin{proof}
By symmetry, we only prove the case for $i=1,j=2.$ There exists $C>0$ such that for any $x\in K$, $v\in X,$
$C^{-1}\|v\|\leq\|d_x f_1 v\|\leq C\|v\|$.
Thus
$
C^{-1}\|d_x f_2^n v\|\leq\|d_{f_2^n x} f_1d_x f_2^n v\|\leq C\|d_x f_2^n v\|.
$
So
$$
\limsup_{n\rightarrow\infty}\frac{1}{n}\log\|d_{f_1 x} f_2^n d_x f_1 v\|=\limsup_{n\rightarrow\infty}\frac{1}{n}\log\|d_{f_2^n x} f_1 d_x f_2^n v\|=\limsup_{n\rightarrow\infty}\frac{1}{n}\log\|d_x f_2^n v\|.
$$

Similarly, there exists $C>1$ such that for any $x\in K$, $C^{-1}\leq\|d_x f_1 \|\leq C$.
Thus
$\|d_x f_2^n\|_{\alpha}\leq\|d_{f_2^n x} f_1d_x f_2^n\|_{\alpha}\leq C\|d_x f_2^n\|_{\alpha}.$
Hence, $l_{\alpha}(x, f_1)=l_{\alpha}(f_2 x, f_1)$.
\end{proof}
\begin{coro}\label{2} For $i,j=1,2$,
\begin{itemize}

\item[(a)] $\Gamma_{f_i}$ are $f_j$-invariant;

\item[(b)] $\lambda_k(x, f_i)$, $m_k(x, f_i)$, $\pi_k(x, f_i)$, $k=1,\ldots, r(x, f_i)$ are $f_j$-invariant;

\item[(c)] $d_x f_j E_k(x, f_i)=E_k(f_j x, f_i)$, $k=1,\ldots, r(x, f_i)$;

\item[(d)] $d_x f_j E_{\alpha}(x, f_i)\subset E_{\alpha}(f_j x, f_i)$.
\end{itemize}
\end{coro}

\begin{proof}[Proof of Theorem A]
For any point $x\in\Gamma_{f_1}$, let
$$
X = E_1(x,f_1) \oplus E_2(x,f_1) \oplus \cdots \oplus E_{r(x,f_1)}(x,f_1) \oplus E_{\alpha}(x,f_1)
$$ be the decomposition for $f_1$. By Corollary \ref{2},
$$d_x f_2 E_k(x, f_1)=E_k(f_2 x, f_1),\;k=1,\ldots, r(x, f_1)
$$ and $d_x f_2 E_{\alpha}(x, f_1)\subset E_{\alpha}(f_2 x, f_1)$. Restricted on $E_k(x, f_1)$ and $E_{\alpha}(x, f_1)$, $\{d_x f_2^n\}$ is a cocycle on $K$ with respect to $f_2$. Now we use Multiplicative Ergodic Theorem (Theorem \ref{MET}) for $E_k(x, f_1)$ and $E_{\alpha}(x, f_1)$ to get subsets $\Gamma^{k}\subset\Gamma_{f_1}$ and $\Gamma^{\alpha}\subset\Gamma_{f_1}$, such that $\mu(\Gamma^{\alpha})=\mu(\Gamma^{k})=1$ for any $\mu\in\mathcal{M}$. Then for any $x\in\Gamma^{k}$ (resp. $x\in\Gamma^{\alpha}$), after relabeling the subscript, if necessary, $E_{k,j_2}(x)$ and $E_{k,r(x,f_2)+1}$ (resp. $E_{r(x,f_1)+1,j_2}(x)$ and $E_{\alpha}(x)$) have desired properties. We take $\Gamma=\cap_{k=1}^{r(x,f_1)}\Gamma^{k}\cap \Gamma^{\alpha}\cap \Gamma_{0}$, then $f_i\Gamma=\Gamma$ for each $i=1,2$.
Then for any $x\in\Gamma$, $E_{j_1,j_2}(x)$, $E_{r(x,f_1)+1,j_2}(x)$, $E_{j_1,r(x,f_2)+1}$ and $E_{\alpha}(x)$ have desired properties and
$$X=\bigoplus_{j_1=1}^{r(x,f_1)}
\bigoplus_{j_2=1}^{r(x,f_2)}E_{j_1,j_2}(x)
\bigoplus_{j_2=1}^{r(x,f_2)}E_{r(x,f_1)+1,j_2}(x)
\bigoplus_{j_1=1}^{r(x,f_1)}E_{j_1,r(x,f_2)+1}\bigoplus E_{\alpha}(x).$$
We now show (\ref{0}) by claiming that for any $\epsilon>0$, $s_1,s_2\in \mathbb{Z}^+$, $1\leq j_i \leq r(x,f_i)$ and $i=1,2$, the set
$$
A_{\epsilon}=\{x\in\Gamma:\exists v_x\in E_{j_1,j_2}(x) \ \mbox{such that} \ \lambda(x,v_x,f_1^{s_1}\circ f_2^{s_2})-s_1\lambda_{j_1}(x,f_1)-s_2\lambda_{j_2}(x,f_2)>4\epsilon\}
$$
satisfies $\mu (A_{\epsilon})=0$ for all $\mu\in\mathcal{M}.$
Suppose it is not true. Then there exists a $\mu\in\mathcal{M}$ with $\mu (A_{\epsilon})>0$. Choose $C > 0$ such that the sets
$$
A_{1}=\{x\in A_{\epsilon}: \|d_x(f_1^{s_1}\circ f_2^{s_2})^nv_x\|\geq C^{-1}\|v_x\|\exp n(\lambda(x,v_x,f_1^{s_1}\circ f_2^{s_2})-\epsilon),\ \forall n\in \mathbb{Z}^+\},
$$
$$
A_{2}=\{x\in A_{\epsilon}:\|d_xf_2^{s_2 n}v\|\leq C\|v\|\exp n(s_2\lambda_{j_2}(x,f_2)+\epsilon),\ \forall v\in E_{j_1,j_2}(x), n\in \mathbb{Z}^+\}
$$
have measures larger than $\frac{1}{2}\mu (A_{\epsilon})$
. Then $\mu(A_{1}\cap A_{2})>0$. By Poincar$\acute{e}$ Recurrence Theorem we can
take $x\in A_{1}\cap A_{2}$ such that there exists a sufficient large integer $n > \frac{2\log C}{\epsilon}$ with $f_1^{s_1 n}x\in A_{1}\cap A_{2}$ and
$$
\|d_xf_1^{s_1 n}v\|\leq C\|v_x\|\exp n(s_1\lambda_{j_1}(x,f_1)+\epsilon),\ \forall v\in E_{j_1,j_2}(x).
$$
Since $d_xf_1^{s_1 n}v\in E_{j_1,j_2}(f_1^{s_1 n}x)$ and $f_1^{s_1 n}x\in A_2$, $\forall v\in E_{j_1,j_2}(x)$,
\begin{eqnarray*}
\|d_x(f_1^{s_1}\circ f_2^{s_2})^nv\|&=&\|d_{f_1^{s_1 n}x}f_2^{s_2n}d_xf_1^{s_1 n}v\|\\
&\leq&C\|d_xf_1^{s_1 n}v\|\exp n(s_2\lambda(f_1^{s_1 n}x,v,f_2)+\epsilon)\\
&\leq&C\|v\|\exp n(s_1\lambda(x,v,f_1)+s_2\lambda(f_1^{s_1 n}x,v,f_2)+2\epsilon),
\end{eqnarray*}
In particular, take $v=v_x$, then
$$
\|d_x(f_1^{s_1}\circ f_2^{s_2})^nv_x\|\leq C^{-1}\|v_x\|\exp n(\lambda(x,v_x,f_1^{s_1}\circ f_2^{s_2})-\epsilon),
$$
 which contradicts the fact $x\in A_{1}$.
Similar claim for the set
$$
B_{\epsilon}=\{x\in\Gamma:\exists v_x\in E_{j_1,j_2}(x) \ \mbox{such that} \ \lambda(x,v_x,f_1^{s_1}\circ f_2^{s_2})-s_1\lambda_{j_1}(x,f_1)-s_2\lambda_{j_2}(x,f_2)<4\epsilon\}
$$
is also true.
Then (\ref{0}) follows by these two claims.
Using the same idea, with some modification, we can prove (\ref{01}).
\end{proof}

\subsection{Proof of Theorem B}

\begin{lemma}\label{t1}
$l_{\alpha}(f)\leq \nu(f_1)l_{\alpha}(f_1)+\nu(f_2)l_{\alpha}(f_2).$
\end{lemma}
\begin{proof}For any $j=1,2$, let $n(\omega,f_j)=\sum_{m=0}^{n-1} \chi_{f_j}f_{\theta^m\omega}$, by Birkhoff Ergodic Theorem $\int\lim_{n\longrightarrow \infty}\frac{1}{n}n(\omega,f_j)d \mathbf{P}_{\nu}=\nu(f_j),$ where $\chi_{f_j}$ is the character function of $f_j$.

For any $x\in \Gamma_f$, $\omega\in\Omega$,
\begin{eqnarray*}
&&\frac{1}{n}\log \sup_{x\in K}\|d_x f(n, \omega)\|_{\alpha}\\
&\leq& \frac{n(\omega,f_1)}{n}\frac{1}{n(\omega,f_1)}\log\sup_{x\in K}\|d_x f_1^{n(\omega,f_1)}\|_{\alpha}+\frac{n(\omega,f_2)}{n}\frac{1}{n(\omega,f_2)}\log\sup_{x\in K}\|d_{x} f_2^{n(\omega,f_2)}\|_{\alpha}.
\end{eqnarray*}
Let $n\rightarrow\infty$, and take integral with respect to $\mathbf{P}_{\nu}$, by Random Subadditive Ergodic Theorem \cite[Theorem 2.2]{K1}, we get the desired result.
\end{proof}

\begin{lemma}\cite[Lemma 3.2]{Hu}
For any $\epsilon>0$, there exists a measurable function $Q:\Gamma\rightarrow[1,\infty)$ such that for any $t_1,t_2\in\mathbb{N}$, $x\in\Gamma$, $0\neq v\in E_{j_1,j_2}(x)$, $1\leq j_1\leq r(x,f_1)$, $1\leq j_2\leq r(x,f_2)$,
\begin{eqnarray*}
&&Q(x)^{-1}\|v\|\exp(t_1\lambda_{j_1}(x, f_1)+t_2\lambda_{j_2}(x, f_2)-3(t_1+t_2)\epsilon)\\
&\leq&\|d_x f_1^{t_1}\circ f_2^{t_2} v\|\\
&\leq& Q(x)\|v\|\exp(t_1\lambda_{j_1}(x, f_1)+t_2\lambda_{j_2}(x, f_2)+3(t_1+t_2)\epsilon).
\end{eqnarray*}
\end{lemma}

\begin{lemma}\label{dengshi}
For any $x\in\Gamma_f\cap\Gamma$, $1\leq j_1\leq r(x,f_1)$, $1\leq j_2\leq r(x,f_2)$, $\mathbf{P}_{\nu}$-a.e. $\omega\in\Omega$,
$$
\lim_{n\rightarrow\infty}\frac{1}{n}\log\|d_xf(n,\omega)|_{E_{j_1,j_2}(x)}\|
=\nu(f_1)\lambda_{j_1}(x,f_1)+\nu(f_2)\lambda_{j_2}(x,f_2).
$$
\end{lemma}
\begin{proof}
For any $1\leq j_1\leq r(x,f_1)$, $1\leq j_2\leq r(x,f_2)$, let $n(\omega,f_j)=\sum_{m=0}^{n-1} \chi_{f_j}f_{\theta^m\omega}$, by Birkhoff Ergodic Theorem $\int\lim_{n\longrightarrow \infty}\frac{1}{n}n(\omega,f_j)d \mathbf{P}_{\nu}=\nu(f_j).$

For any  $\epsilon>0$, $x\in \Gamma_f$, $\omega\in\Omega$, $v\in E_{j_1,j_2}$,
\begin{eqnarray*}
&&\frac{1}{n}(\log Q(x)^{-1}+ \log \|v\|)+ \frac{n(\omega,f_1)}{n}\lambda_{j_1}(x, f_1)+\frac{n(\omega,f_2)}{n}\lambda_{j_2}(x, f_2)-3\epsilon\\
&\leq&\frac{1}{n}\log\|d_xf(n,\omega)v\|\\
&\leq&\frac{1}{n}(\log Q(x)+ \log \|v\|)+ \frac{n(\omega,f_1)}{n}\lambda_{j_1}(x, f_1)+\frac{n(\omega,f_2)}{n}\lambda_{j_2}(x, f_2)+3\epsilon.
\end{eqnarray*}

Let $n\rightarrow\infty$, and take integral with respect to $\mathbf{P}_{\nu}$, by Random Subadditive Ergodic Theorem \cite[Theorem 2.2]{K1}, we get the desired result.
\end{proof}
By a similar argument, we have the following corollary.
\begin{coro}\label{jia1}
For any $x\in\Gamma_f\cap\Gamma$, $1\leq j_1\leq r(x,f_1)$, $1\leq j_2\leq r(x,f_2)$, $\mathbf{P}_{\nu}$-a.e. $\omega\in\Omega$,
$$
\lim_{n\rightarrow\infty}\frac{1}{n}\log\|d_xf(n,\omega)|_{E_{r(x,f_1)+1,j_2}}\|
\leq\nu(f_1)\lambda_{\alpha}+\nu(f_2)\lambda_{j_2}(x,f_2)
$$
and
$$
\lim_{n\rightarrow\infty}\frac{1}{n}\log\|d_xf(n,\omega)|_{E_{j_1,r(x,f_2)+1}}\|
\leq\nu(f_1)\lambda_{j_1}(x,f_1)+\nu(f_2)\lambda_{\alpha}.
$$
\end{coro}

Since $X$ is a separable  Hilbert space, we can
characterize  $V_{p}(\cdot)$ in
(\ref{tiji})  by exterior power. Denote by $X^{\wedge p}$ the $p$-th
exterior power space of $X$, i.e., the collection of completely
antisymmetric elements of the Hilbert space of tensor product of $p$
copies of $X$.  Let  $\{\xi_{i}\}_{i=1}^{\infty}$ be a countable
orthonormal base  of $X$. Then
\[
\{\xi_{i_1}\wedge \cdots \wedge \xi_{i_p}:\ \ 1\leq i_1<i_2<\cdots
<i_p<\infty\}
\]
is a basis of $X^{\wedge p}$. Define an inner product $\left<\cdot,
\cdot \right>$ on $X^{\wedge p}$ by letting
\[
\left<\xi_{i_1}\wedge \cdots \wedge \xi_{i_p}, \xi_{j_1}\wedge
\cdots \wedge \xi_{j_p}\right>=\left\{
                                 \begin{array}{ll}
                                   1, & \hbox{if}\ (i_1,\cdots, i_p)=(j_1,\cdots, j_p); \\
                                   0, & \hbox{otherwise.}
                                 \end{array}
                               \right.
\]
This inner product is independent of the choice of the orthonormal
basis $\{\xi_i\}_{i=1}^{\infty}$. Denote by $|\cdot|$ the norm on
$X^{\wedge p}$ induced by this inner product.  For any vectors $\xi_{1},\cdots, \xi_p$ of $X$,  $|\xi_{1}\wedge \cdots \wedge
\xi_{p}|$ is just the volume of the parallelotope formed by these vectors and its square is the classical Gram determinant of the vectors.
So we have

\begin{lemma}\cite[Lemma 2.3]{l1}\label{tiji1}
$V_{p}(\xi_{1}, \cdots, \xi_{p})=|\xi_{1}\wedge \cdots \wedge
\xi_{p}|$, for $\xi_{1}\wedge \cdots \wedge \xi_{p}\in X^{\wedge
p}$.
\end{lemma}
Given a $C^1$ map $T:\  U\to X$ for some open subset $U$ of $X$,
define for $x\in U$, $p\in \Bbb N$,
\begin{eqnarray*}
  (d_x T)^{\wedge p}: && \ X^{\wedge p}\to X^{\wedge p}\\
&& \xi_{1}\wedge \cdots \wedge \xi_{p}\to (d_x T)\xi_{1}\wedge
\cdots \wedge (d_x T)\xi_{p}.
\end{eqnarray*}

It is true by Lemma \ref{tiji1} that
\begin{eqnarray*}
  |(d_x T)^{\wedge p}|=\sup_{\|\xi_i\|=1, \ 1\leq i\leq p} V_{p}((d_x T)\xi_1,\cdots, (d_x T)
\xi_p)=V_{p}(d_x T).
\end{eqnarray*}
As a consequence, we have for $f$ as in
Theorem \ref{thm2.1}, if $X$ is a separable Hilbert space, then for
$\mathbf{P}_{\nu}$-a.e $\omega\in \Omega$, $x\in \Gamma_f$ and $p\leq \Sigma_{j=1}^{r(x,f)} m_j(x,f),$
\begin{equation}\label{tiji2}
\lim\limits_{n\rightarrow \infty}\frac{1}{n}\log \left|(d_x f(n,
\omega))^{\wedge p}\right|=\sum_{k=1}^{p}\widetilde{\lambda}_k (x,f),
\end{equation}
where $\{\widetilde{\lambda}_k(x,f)\}$ are $\lambda_j(
    x,f)$'s repeated with multiplicity $m_j(x,f)$.
Put
$$
\lambda_{\alpha}=0>\max\{l_{\alpha}(f_1), l_{\alpha}(f_2)\},
 $$
 $$
 E^u(x, f_1):=\oplus_{j=1}^{r(x,f_1)} E_j(x, f_1),\;\;E^u(x, f_2):=\oplus_{j=1}^{r(x,f_2)} E_j(x, f_2),
 $$
and
$$
M(x,f_1):=\sum_{j=1}^{r(x,f_1)}m_j(x,f_1),\,\,M(x,f_2):=\sum_{j=1}^{r(x,f_2)}m_j(x,f_2).
 $$
Similarly, let
$$
E^u(\omega, x, f):=\oplus_{j=1}^{r(x,f)} E_i(\omega,x, f)\;\text{and}\; M(x,f):=\sum_{j=1}^{r(x,f)}m_j(x,f).
$$
%Since $
%X=\bigoplus_{j_1=1}^{r(x,f_1)}
%\bigoplus_{j_2=1}^{r(x,f_2)}E_{j_1,j_2}(x)
%\bigoplus_{j_2=1}^{r(x,f_2)}E_{r(x,f_1)+1,j_2}(x)
%\bigoplus_{j_1=1}^{r(x,f_1)}E_{j_1,r(x,f_2)+1}\bigoplus E_{\alpha}(x)$ and
%Next we need a detailed analysis of  $E^u(x, f_1)$ and $E^u(x, f_2).$
%First, it is clear that
%By Lemma \ref{t1} and assumption (H1),
Next, let $
\underline{E^{u}}(x):=\bigoplus_{j_1=1}^{r(x,f_1)}
\bigoplus_{j_2=1}^{r(x,f_2)}E_{j_1,j_2}(x)
$, $
\underline{M(x)}:=dim \underline{E^{u}}(x),$$$
\overline{E^{u}}(x):=\bigoplus_{j_1=1}^{r(x,f_1)}
\bigoplus_{j_2=1}^{r(x,f_2)}E_{j_1,j_2}(x)\bigoplus_{j_2=1}^{r(x,f_2)}E_{r(x,f_1)+1,j_2}(x)
\bigoplus_{j_1=1}^{r(x,f_1)}E_{j_1,r(x,f_2)+1},
$$
and $\overline{M(x)}:=dim \overline{E^{u}}(x).$
Moreover, $d_{j_1}(x,f_1)=m_{j_1}(x,f_1)-dim E_{j_1,r(x,f_2)+1}(x)$, $d_{j_2}(x,f_2)=m_{j_2}(x,f_2)-dim E_{r(x,f_1)+1,j_2}(x)$.
\begin{lemma}\label{srb}
Let $f$ be a finitely generated random transformation of an infinite dimensional Banach
  space $X$ over $(\Omega, \mathcal{A}, \mathbf{P}_{\nu}, \theta)$. Suppose $\mu\in \mathcal{M}$ and $(f_i,\mu)$ satisfies (H0-H1), then  $$\int \sum_{i=1}^2\sum_{\lambda_{j_k}(x,f_i)>0}\nu(f_i)\lambda_{j_k}(x,f_i)d_{j_k}(x,f_i)d\mu\leq$$
$$\int\log|\mbox{det}(D_x f_{w}|_{E^u(w, x)})|d\mathbf{P}_{\nu}\times\mu\leq \int \sum_{i=1}^2\sum_{\lambda_{j_k}(x,f_i)>0}\nu(f_i)\lambda_{j_k}(x,f_i)m_{j_k}(x,f_i)d\mu.
$$
\end{lemma}
\begin{proof}
By Lemma \ref{dengshi} and Corollary \ref{jia1}, it is clear that $\underline{E^{u}}(x)\subset E^u(\omega, x, f)\subset \overline{E^{u}}(x)$, in particular, $\underline{M(x)}\leq M(x,f)\leq \overline{M(x)}$, for any $x\in\Gamma\cap\Gamma_f$, $\mathbf{P}_{\nu}$-a.s. Therefore, combining with Birkhoff ergodic theory, we have
\begin{eqnarray*}
\int\log|\mbox{det}(D_x f_{w}|_{E^u(w, x)})|d\mathbf{P}_{\nu}\times\mu
&=&\int\lim\limits_{n\rightarrow \infty}\frac{1}{n}\log \left|(d_x f(n,
\omega))^{\wedge M(x,f)}\right|d\mathbf{P}_{\nu}\times\mu\\
&\leq&\int\lim\limits_{n\rightarrow \infty}\frac{1}{n}\log \left|(d_x f(n,
\omega))^{\wedge \overline{M(x)}}\right|d\mathbf{P}_{\nu}\times\mu\\
&=&\int \sum_{i=1}^2\sum_{\lambda_{j_k}(x,f_i)>0}\nu(f_i)\lambda_{j_k}(x,f_i)m_{j_k}(x,f_i)d\mu.
\end{eqnarray*}
Similarly,
\begin{eqnarray*}
\int\log|\mbox{det}(D_x f_{w}|_{E^u(w, x)})|d\mathbf{P}_{\nu}\times\mu
&=&\int\lim\limits_{n\rightarrow \infty}\frac{1}{n}\log \left|(d_x f(n,
\omega))^{\wedge M(x,f)}\right|d\mathbf{P}_{\nu}\times\mu\\
&\geq&\int\lim\limits_{n\rightarrow \infty}\frac{1}{n}\log \left|(d_x f(n,
\omega))^{\wedge \underline{M(x)}}\right|d\mathbf{P}_{\nu}\times\mu\\
&=&\int \sum_{i=1}^2\sum_{\lambda_{j_k}(x,f_i)>0}\nu(f_i)\lambda_{j_k}(x,f_i)d_{j_k}(x,f_i)d\mu.
\end{eqnarray*}
\end{proof}

To prove Theorem B, we need to
establish the relation between local covering numbers of tangent maps of $f$
and Lyapunov exponents of generators.
For $A\subset X$, $\epsilon>0$, define
\begin{eqnarray*}
  r(A, \epsilon, d)=\inf \{n\geq 1:&&\text{there exist }  (x_1,\cdots, x_n)\in
  X^{n}\text{ and } (\epsilon_1,\cdots, \epsilon_n)\in \Bbb
  {R^+}^{n}\\
&& \mbox{such that}\ \  A\subset \cup_{i=1}^{n}B(x_i, \epsilon_i),\
\epsilon_i<\epsilon\}.
\end{eqnarray*}
For $T\in L(X)$, $\epsilon>0$, let
\[
R(T, \epsilon):=r(T(B_{X}), \epsilon, d),
\]
where $B_{X}$ denotes the unit ball in $X$. Let $\beta>0$. For $\omega\in
\Omega$ and $x\in K$, define
\[
\Delta_{\omega}^{\beta}(x,
f):=\lim\limits_{n\rightarrow\infty}\frac{1}{n}\log R(d_{x}f(n, \omega),
e^{-n\beta})
\]
whenever the limit exists. By \cite[Proposition 3.4]{l0}, the limit exists $\mathbf{P}_{\nu}\times\mu$-a.s.
\begin{lemma}\label{Lb}
Let $f$ be a finitely generated random transformation of an infinite dimensional Hilbert
space $X$ over $(\Omega, \mathcal{A}, \mathbf{P}_{\nu}, \theta)$. Suppose $\mu\in \mathcal{M}$ and $(f_i,\mu)$ satisfies (H0-H1), and $0<\beta<-\max\{l_{\alpha}(f_1),l_{\alpha}(f_2)\}$. Then for $\mathbf{P}_{\nu}\times\mu$-a.e. $(\omega,x)$,
\begin{eqnarray*}
&&\Delta_{\omega}^{\beta}(x, f)\leq\sum_{j_1=1}^{r(x,f_1)}\sum_{j_2=1}^{r(x,f_2)}[\nu(f_1)(\lambda_{j_1}(x,f_1)+\beta)+\nu(f_1)(\lambda_{j_1}(x,f_1)+\beta)]^+ m_{j_1,j_2}(x)\\
&+&\sum_{j_2=1}^{r(x,f_2)}\nu(f_1)(\lambda_{j_1}(x,f_1)+\beta)^+m_{r(x,f_1)+1,j_2}+\sum_{j_2=1}^{r(x,f_2)}\nu(f_2)(\lambda_{j_2}(x,f_2)+\beta)^+m_{j_1,r(x,f_2)+1}.
\end{eqnarray*}
 and
  \begin{equation}\label{cover-lya2}
\Delta_{\omega}^{\beta}(x, f)\geq\sum_{i=1}^2\sum_{k=1}^{r(x,f_i)}\nu(f_i)(\lambda_{j_k}(x,f_i)+\beta)^+d_{j_k}(x,f_i).
  \end{equation}
\end{lemma}
\begin{proof}
  Let $0<\beta<-\-\max\{l_{\alpha}(f_1),l_{\alpha}(f_2)\}$. For $x\in \Gamma_f\cap\Gamma$, $i=1,2$, let $r(x,f_i)$ be the maximal number such
that $\lambda_{r(x,f_i)}(x,f_i)\geq -\beta$. By (H1) and Theorem A, consider the decomposition
%$E_{r(x,f_1)+1,j_2}(x)=E_{r(x,f_1)+1,j_2}(x)=\{0\}$, for any $1\leq j_1\leq r(x,f_1)$, $1\leq j_2\leq r(x,f_2)$.
$$
X=\bigoplus_{j_1=1}^{r(x,f_1)}
\bigoplus_{j_2=1}^{r(x,f_2)}E_{j_1,j_2}(x)
\bigoplus_{j_2=1}^{r(x,f_2)}E_{r(x,f_1)+1,j_2}(x)
\bigoplus_{j_1=1}^{r(x,f_1)}E_{j_1,r(x,f_2)+1}\bigoplus E_{\alpha}(x)
$$
with $\pi_{j_1,j_2}, \pi_{r(x,f_1)+1,j_2}, \pi_{j_1,r(x,f_2)+1},\pi_{\alpha}$ being the family
of associated projections. Then
\begin{eqnarray*}
&&B_{E}\subset
\oplus_{j_1=1}^{r(x,f_1)}
\oplus_{j_2=1}^{r(x,f_2)}|\pi_{j_1,j_2}|B_{j_1,j_2}\oplus_{j_2=1}^{r(x,f_2)}|\pi_{r(x,f_1)+1,j_2}|B_{r(x,f_1)+1,j_2}(x)
\\
&&\oplus_{j_1=1}^{r(x,f_1)}|\pi_{j_1, r(x,f_2)+1}|B_{j_1,r(x,f_2)+1}\oplus |\pi_{\alpha}| B_{E_{\alpha}}.
\end{eqnarray*}

Consider $d_{x}f(n,\omega)(B_{E})$. On the one hand,
\begin{eqnarray*}
  d_x f(n,\omega)(B_{E}) \subset \oplus_{j_1=1}^{r(x,f_1)}
\oplus_{j_2=1}^{r(x,f_2)}|\pi_{j_1,j_2}|d_xf(n,\omega)(B_{j_1,j_2}) \\\oplus_{j_2=1}^{r(x,f_2)}|\pi_{r(x,f_1)+1,j_2}|d_xf(n,\omega)(B_{r(x,f_1)+1,j_2})
\oplus
 |\pi_{\alpha}|d_x f(n,\omega)(B_{E_{\alpha}}).
\end{eqnarray*}
So for $\beta<\gamma\leq-\-\max\{l_{\alpha}(f_1),l_{\alpha}(f_2)\}$, if we
choose $n$ large such that
\[
\|d_x
f(n,\omega)\|_{\alpha}<e^{-n\gamma}<(\sum_{j_1=1}^{r(x,f_1)}
\sum_{j_2=1}^{r(x,f_2)}|\pi_{j_1,j_2}|+\sum_{j_1=1}^{r(x,f_1)}|\pi_{j_1,r(x,f_2)+1}|+\sum_{j_2=1}^{r(x,f_2)}|\pi_{r(x,f_1)+1,j_2}|+|\pi_{\alpha}|)^{-1}e^{-n\beta},
\]
then
$$
r(d_xf(n,\omega)\left(B_{E}\right), e^{-n\beta})\leq
\prod_{j_1=1}^{r(x,f_1)}\prod_{j_2=1}^{r(x,f_2)}r(d_xf(n,\omega)(B_{E_{j_1,j_2}}), e^{-n\gamma})\cdot$$$$\prod_{j_1=1}^{r(x,f_1)}r(d_xf(n,\omega)(B_{E_{j_1,r(x,f_2)+1}}),e^{-n\gamma})\cdot\prod_{j_2=1}^{r(x,f_2)}r(d_xf(n,\omega)(B_{E_{r(x,f_1)+1,j_2}}),e^{-n\gamma}).
$$
For each $1\leq j_1\leq r(x,f_1)$, $1\leq j_2\leq r(x,f_2)$, we
have
\begin{eqnarray*}
r(d_xf(n,\omega)(B_{E_{j_1,j_2}}), e^{-n\gamma})\leq \{[m_{j_1,j_2}\cdot
\|d_xf(n,\omega)|_{E_{j_1,j_2}}\|\cdot e^{n\gamma}]+1\}^{m_{j_1,j_2}},\\
r(d_xf(n,\omega)(B_{E_{j_1,r(x,f_2)+1}}), e^{-n\gamma})\leq \{[m_{j_1,r(x,f_2)+1}\cdot
\|d_xf(n,\omega)|_{E_{j_1,r(x,f_2)+1}}\|\cdot e^{n\gamma}]+1\}^{m_{j_1,r(x,f_2)+1}},\\
r(d_xf(n,\omega)(B_{E_{r(x,f_1)+1,j_2}}), e^{-n\gamma})\leq \{[m_{r(x,f_1)+1,j_2}\cdot
\|d_xf(n,\omega)|_{E_{r(x,f_1)+1,j_2}}\|\cdot e^{n\gamma}]+1\}^{m_{r(x,f_1)+1,j_2}},
\end{eqnarray*}
where $[a]$ denotes the integer part of the number $a$, $m_{j_1,j_2}=dim(E_{j_1,j_2})$, $m_{r(x,f_1)+1,j_2}=dim(E_{r(x,f_1)+1,j_2})$ and $m_{j_1,r(x,f_2)+1}=dim(E_{j_1,r(x,f_2)+1})$. From this,
we deduce that
\begin{eqnarray*}
&&\varlimsup_{n\rightarrow\infty}\frac{1}{n}\log R(d_{x}f(n, \omega),
e^{-n\beta})\\
&\leq&\sum_{j_1=1}^{r(x,f_1)}\sum_{j_2=1}^{r(x,f_2)}[\nu(f_1)(\lambda_{j_1}(x,f_1)+\gamma)+\nu(f_1)(\lambda_{j_1}(x,f_1)+\gamma)]^+ m_{j_1,j_2}(x)\\
&+&\sum_{j_2=1}^{r(x,f_2)}\nu(f_1)(\lambda_{j_1}(x,f_1)+\gamma)^+m_{r(x,f_1)+1,j_2}+\sum_{j_2=1}^{r(x,f_2)}\nu(f_2)(\lambda_{j_2}(x,f_2)+\gamma)^+m_{j_1,r(x,f_2)+1}.
\end{eqnarray*}

Since $\gamma>\beta$ is arbitrary, we have the first inequality.

For the other inequality, let $\gamma$ be such that
$\max\{-\lambda_{r(x,f_1)},-\lambda_{r(x,f_2)}\}<\gamma<\beta$. Let $n$ be large such that
\[
\|d_x
f(n,\omega)\|_{\alpha}<e^{-n\beta}<(2r(x,f_1)\times r(x,f_2)\sum_{j_1=1}^{r(x,f_1)}
\sum_{j_2=1}^{r(x,f_2)}|\pi_{j_1,j_2}|)^{-1}e^{-n\gamma}.
\]
Since $\frac{1}{r(x,f_1)\times r(x,f_2)}(\bigoplus_{j_1=1}^{r(x,f_1)}
\bigoplus_{j_2=1}^{r(x,f_2)}B_{j_1,j_2})\subset B_{E}$, we
have
\[
r(d_xf(n,\omega)(B_{E}), e^{-n\beta})\geq
\prod_{j_1=1}^{r(x,f_1)}\prod_{j_2=1}^{r(x,f_2)}
S(d_xf(n,\omega)(B_{E_{j_1,j_2}}), e^{-n\gamma}),
\]
where $S(A, \delta)$ is the maximal number of subcollection of $A$
such that any two points of it has distance at least $\delta$. Now
for $1\leq j_1\leq r(x,f_1)$, $1\leq j_2\leq r(x,f_2)$,
\[
S(d_xf(n,\omega)(B_{E_{j_1,j_2}}), e^{-n\gamma})\geq
\max\{(2e^{n\gamma}m_{j_1,j_2}^{-1}\|
d_xf(n,\omega)|_{E_{j_1,j_2}}^{-1}\|^{-1})^{m_{j_1,j_2}},
1\}.
\]
From this we deduce that
\[
\varliminf_{n\rightarrow\infty}\frac{1}{n}\log R(d_xf(n,\omega),
e^{-n\beta})\geq \sum_{i=1}^2\sum_{k=1}^{r(x,f_i)}\nu(f_i)(\lambda_{j_k}(x,f_i)+\gamma)^+d_{j_k}(x,f_i).
\]
Since $\gamma<\beta$ is arbitrary,
(\ref{cover-lya2}) is also proved. We are done.
\end{proof}
\begin{proof}[Proof of Theorem B]The main strategy of the proofs are making necessary modifications by comparing the dynamics of $f$ and the generators. We indicate that Lemma \ref{Lb} is the key step to establish the relation between Lyapunov exponents of $f$ and those of its generators $f_1$ and $f_2$.
So we will concentrate upon the necessary
modifications and omit most of the parallel arguments, for which we refer
the reader to \cite{l0}. By ergodic decomposition theorem \cite[Theorem 1.1]{LQ95}, we restrict ourselves to the case $\mu\in \mathcal{M}_f^e.$
Let $0<\beta<-\-\max\{l_{\alpha}(f_1),l_{\alpha}(f_2)\}$.
For each $k\in \mathbb{N}$, let
\begin{equation}\label{ak}
A_k:=\{\omega\in \Omega:\ \ \frac{1}{n}\log \sup_{x\in K}\|d_x f(n,\omega)\|_{\alpha}<\frac{1}{2}(l_{\alpha}-\beta)\ \mbox{for}\  n\geq
k\}.
\end{equation}
$\mathbf{P}_{\nu}(A_{k})$ increases to $1$ as $k$ goes to
infinity. For $k\in \Bbb N$, define
\begin{equation*}
  f_k(\omega, x)=\left\{
              \begin{array}{ll}
                \log R(d_x f(k, \omega), e^{-k\beta}), & \hbox{if}\ \omega\in A_k, \ x\in K; \\
                0, & \hbox{otherwise.}
              \end{array}
            \right.
\end{equation*}
Then by Lemma \ref{Lb} and similar argument in \cite[Lemma 3.5]{l0}, for  $\mathbf{P}_{\nu}\times\mu\mbox{-a.e.}$ $(\omega, x)$,
\begin{eqnarray*}
&&\lim\limits_{k\rightarrow \infty}\frac{1}{k}f_k(\omega, x)\\
&\leq&\sum_{j_1=1}^{r(x,f_1)}\sum_{j_2=1}^{r(x,f_2)}[\nu(f_1)(\lambda_{j_1}(x,f_1)+\beta)+\nu(f_1)(\lambda_{j_1}(x,f_1)+\beta)]^+ m_{j_1,j_2}(x)\\
&+&\sum_{j_2=1}^{r(x,f_2)}\nu(f_1)(\lambda_{j_1}(x,f_1)+\beta)^+m_{r(x,f_1)+1,j_2}+\sum_{j_2=1}^{r(x,f_2)}\nu(f_2)(\lambda_{j_2}(x,f_2)+\beta)^+m_{j_1,r(x,f_2)+1}.
\end{eqnarray*}

Let $\mathcal{P}$ be a finite measurable partition of
$\Omega\times X$. For $k\in \Bbb N$, define a function $A(f_k, \mathcal{P}):\
\Omega\times X\to \Bbb R$ by letting
\begin{equation*}
  A(f_k, \mathcal{P})(\omega, x)=\sum_{P\in \mathcal{P}^w}\left(\sup_{x\in P}
f_k(\omega,x)\right)\cdot \chi_{P}(x),
\end{equation*}
where $\chi_{P}$ is the characteristic function of the set $P$.  We
have the following  relation between $f_k$ and
$A(f_k,\mathcal{P})$ \cite[Proposition 3.7]{l0} :
\begin{equation}\label{step1}
\lim\limits_{n\rightarrow \infty}\frac{1}{n}f_n(\omega,
x)=\varlimsup_{m\rightarrow \infty}\varlimsup_{n\rightarrow
\infty}\frac{1}{n}A(f_n, \widetilde{\mathcal{P}}_{-n}^{m})(\omega, x),\
\mathbf{P}_{\nu}\times\mu\mbox{-a.e.}
\end{equation}
We begin with the selection of a sequence of ``good'' sets $A_{k,
l}^{m}$.  Let $0<\beta<-\frac{1}{5}\max\{l_{\alpha}(f_1),l_{\alpha}(f_2)\}$ be fixed. Let $\{\widetilde{\mathcal{P}}^m\}_{m\in \Bbb N}$ be as above so that (\ref{step1})
holds. Let
\begin{eqnarray*}
&&\Delta=\sum_{j_1=1}^{r(x,f_1)}\sum_{j_2=1}^{r(x,f_2)}[\nu(f_1)(\lambda_{j_1}(x,f_1)+\beta)+\nu(f_1)(\lambda_{j_1}(x,f_1)+\beta)]^+ m_{j_1,j_2}(x)\\
&+&\sum_{j_2=1}^{r(x,f_2)}\nu(f_1)(\lambda_{j_1}(x,f_1)+\beta)^+m_{r(x,f_1)+1,j_2}+\sum_{j_2=1}^{r(x,f_2)}\nu(f_2)(\lambda_{j_2}(x,f_2)+\beta)^+m_{j_1,r(x,f_2)+1}.
\end{eqnarray*}

For $m\in \Bbb N$, consider
\[
D^m:=\{(\omega, x):\ \lim\limits_{n\rightarrow \infty} \frac{1}{n}A(f_n,
\widetilde{\mathcal{P}}_{-n}^{m})(\omega, x)\leq \Delta+\frac{1}{2}\beta\}.
\]
It is clear that $\mathbf{P}_{\nu}\times\mu(D^m)$ tends to $1$ as $m$ tends to infinity.
For $k\in \Bbb N$, let
\[
D^{m}_{k}:=\{(\omega, x):\ \frac{1}{n}A(f_n,
\widetilde{\mathcal{P}}_{-n}^{m})(\omega,x)\leq \Delta+\beta\ \mbox{for}\  n\geq
k\}.
\]
Then $\mathbf{P}_{\nu}\times\mu(D^{m}_{k})$ increases to  $\mathbf{P}_{\nu}\times\mu(D_{m})$ as $k$ increases to infinity. For  $k\in \Bbb N$, let $A_k$  be  as in (\ref{ak}). For $l\in \Bbb N$, define
\begin{eqnarray*}
  A_{k, l}:=\{ \omega\in A_{k}: && \|f(k, \omega)x-f(k, \omega)y-d_x f(\omega,
  k)(x-y)\|\leq e^{-k\beta}\epsilon\notag\\
  && \mbox{for}\ x, y\in K,\ \|x-y\|\leq \epsilon,\
  \epsilon<\epsilon_0/l\}.\label{equ3.6}
\end{eqnarray*}
Since $ f(k, \omega)$ is $C^1$ in a neighbourhood of $K$ and
$K$ is compact, we see that $\mathbf{P}_{\nu}(A_{k, l})$ increases to
$\mathbf{P}_{\nu}(A_{k})$ as $l$ goes to infinity and hence
\[
\lim\limits_{k\rightarrow \infty}\lim\limits_{l\rightarrow
\infty}\int_{\Omega\backslash A_{k, l}}\log^{+}\sup\limits_{x\in
B(K, \epsilon_0)}\|d_x f_{\omega}\|\ d\mathbf{P}_{\nu}(w)=0.
\]
Fix $A_{k, l}$ and for
$m\in \Bbb N$, define
\[
A_{k, l}^{m}:=\{(\omega, x)\in D_{k}^{m}:\ \ \omega\in A_{k, l}\}.
\]
We have
\begin{eqnarray*}
  &&\lim\limits_{m\rightarrow \infty}\lim\limits_{k\rightarrow
\infty}\lim\limits_{l\rightarrow \infty}\mathbf{P}_{\nu}\times\mu(A_{k, l}^{m})=1;\\
&& \lim\limits_{m\rightarrow \infty}\lim\limits_{k\rightarrow
\infty}\lim\limits_{l\rightarrow \infty}\int_{(\Omega\times
X)\backslash A_{k, l}^{m}}\log^+\sup_{x\in B(K, \epsilon_0)}\|d_x
f_{\omega}\|\ d\mathbf{P}_{\nu}\times\mu(\omega, x)=0.
\end{eqnarray*}
By \cite[Lemma 3.8]{l0},
\[
h_{\mu}(f)=\lim\limits_{m\rightarrow \infty}\lim\limits_{k\rightarrow
\infty}\lim\limits_{l\rightarrow
\infty}\lim\limits_{\epsilon\rightarrow 0}\varliminf_{n\rightarrow
\infty}-\frac{1}{n}\log \mu_{w}\left(B_{A_{k, l}^{m}, n}^{\omega}(x,
\epsilon)\right), \ \mathbf{P}_{\nu}\times\mu\mbox{-a.e.}
\]
where \begin{eqnarray*}
B_{A, n}^{\omega}(x, \epsilon):=\{ y\in K: && \mbox{for}\  0\leq
i\leq n-1,\ F^i(\omega,x)\in A\
\mbox{iff}\ F^i(\omega, y)\in A\notag\\
&& \mbox{and} \ d(f(i, \omega)x, f(i, \omega)y)<\epsilon\ \mbox{if}\
F^i(\omega,x)\in A\}.
\end{eqnarray*}
The proof is thus finished by \cite[Proposition 3.9]{l0},
since there exists $N_0\in \Bbb
N$ such that for $\mathbf{P}_{\nu}\times\mu$-a.e. $(\omega, x)$, the inequality
\[
\lim\limits_{\epsilon\rightarrow 0}\varliminf_{n\rightarrow
+\infty}-\frac{1}{n}\log \mu\left(B_{A_{k, l}^{m}, n}^{\omega}(x,
\epsilon)\right)\leq \Delta+3\beta
\]
holds for any $k, l, m\geq N_0$.
\end{proof}

\subsection{Proof of Theorem C}

\begin{proof}[Proof of Theorem C]
Again we will concentrate upon the necessary
modifications and omit most of the parallel arguments, for which we refer
the reader to \cite{l1}.

Let $\eta$ be a partition subordinate to the unstable manifolds as in \cite[Proposition 2.9]{l1}. Since, for any $n\in \Bbb N$,
$
\frac{1}{n}H_{\mathbf{P}_{\nu}\times\mu}(F^{-n}\eta|\eta)
= H_{\mathbf{P}_{\nu}\times\mu}(F^{-1}\eta|\eta)$, we have
	\begin{eqnarray*}
	 H_{\mathbf{P}_{\nu}\times\mu}(F^{-1}\eta|\eta)= \lim\limits_{n\rightarrow \infty}\frac{1}{n} H_{\mathbf{P}_{\nu}\times\mu}(F^{-n}\eta|\eta)\leq h_{\mu}(f,\eta)\leq h_{\mu}(f).
	\end{eqnarray*}
Notice that
$$
  H_{\mathbf{P}_{\nu}\times\mu}(F^{-1}\eta|\eta)=-\int\log \mu^{\eta^{\omega} (x)}((F^{-1}\eta)^{\omega}(x))\ d\mathbf{P}_{\nu}\times\mu.
$$
Assume $\mu$ satisfies SRB property, then
$$
-\int\log \mu^{\eta^{\omega} (x)}((F^{-1}\eta)^{\omega}(x))\ d\mathbf{P}_{\nu}\times\mu=\int\log|\mbox{det}(d_x f_{\omega}|_{E^u(\omega, x)})|d\mathbf{P}_{\nu}\times\mu.
$$
By Lemma \ref{srb} and Theorem A, we obtain the desired inequality.
\end{proof}
\begin{proof}[Proof of Corollary \ref{c1}]
 By Lemma \ref{t1}, Lemma \ref{srb} and assumption (H3),
$$
E^u(x):=E^u(x, f_1)=E^u(x, f_2)=\bigoplus_{j_1=1}^{r(x,f_1)}
\bigoplus_{j_2=1}^{r(x,f_2)}E_{j_1,j_2}(x).
$$
Thus,  \begin{equation*}
\int\log|\mbox{det}(D_x f_{w}|_{E^u(w, x)})|d\mathbf{P}_{\nu}\times\mu\geq \int \sum_{i=1}^2\sum_{\lambda_{j_k}(x,f_i)>0}\nu(f_i)\lambda_{j_k}(x,f_i)m_{j_k}(x,f_i)d\mu.
\end{equation*}
Therefore, combining with Theorem C, we get the desired result.
\end{proof}

\subsection{Proof of Theorem D}

\begin{proof}[Proof of Theorem D]
By \cite{Geller}, $F$ is an extension of $\sigma_{\mathbf{f}}$ since we can define a map
$$
\tilde{\pi}: \Omega\times K\longrightarrow K_{\mathbf{f}},\;\tilde{\pi}(\omega,x)=\{f(n,\omega)(x)\}_{n\in{\mathbb{N}}}
$$
such that $\tilde{\pi}\circ F=\sigma_{\mathbf{f}} \circ \tilde{\pi}$.
Therefore, $h(\sigma_{\mathbf{f}})\leq h(F)$.

By Abromov-Rohklin formula \cite{Abromov}, for any invariant measure of $F$ in the form of $\mathbf{P}_{\nu}\times\mu$, we have
$$
h_{\mathbf{P}_{\nu}\times\mu}(F)=h_{\mathbf{P}_{\nu}}(\theta) +h_{\mu}(f).
$$
Since $h_{\mathbf{P}_{\nu}}(\theta)= -\sum_{i=1}^2\nu(f_i)\log\nu(f_i)$, by (\ref{m2}),
\begin{equation}\label{Abromov}
h_{\mathbf{P}_{\nu}\times\mu}(F)=-\sum_{i=1}^2\nu(f_i)\log\nu(f_i) +\int \sum_{i=1}^2\sum_{\lambda_{j_k}(x,f_i)>0}\nu(f_i)\lambda_{j_k}(x,f_i)m_{j_k}(x,f_i)
d\mu.
\end{equation}
By the assumption on the measure with maximal entropy of $F$ and the fact $h(\sigma_{\mathbf{f}})\leq h(F)$,  we obtain (\ref{Fried1}).

When $\mu$ is ergodic, (\ref{Abromov}) becomes
\begin{equation}\label{Fried4}
h_{\mathbf{P}_{\nu}\times\mu}(F)=-\sum_{i=1}^2\nu(f_i)\log\nu(f_i) + \sum_{i=1}^2\sum_{\lambda_{j_k}(x,f_i)>0}\nu(f_i)\lambda_{j_k}(x,f_i)m_{j_k}(x,f_i)
.
\end{equation}
Define a function
$$
J:\Omega\longrightarrow \mathbb{R}^+,\;
J(\omega)=\sum_{i=1}^2\sum_{\lambda_{j_k}(x,f_{\omega})>0}\nu(f_{\omega})\lambda_{j_k}(x,f_{\omega})m_{j_k}(x,f_{\omega}).
$$
Then (\ref{Fried4}) becomes
\begin{equation}\label{Fried5}
h_{\mathbf{P}_{\nu}\times\mu}(F)= -\sum_{i=1}^2\nu(f_i)\log\nu(f_i) +\int_{\Omega}J d\mathbf{P}_{\nu}(\omega).
\end{equation}
 Since $\mathbf{P}_{\nu}\times\mu$ is a measure with maximal entropy of $F$, we can apply the variational principle of $F$ as follows
\begin{eqnarray}\label{pressure}
h(F)&=&\sup_{\mathbf{P}_{\nu'}}\Big\{h_{\mathbf{P}_{\nu'}}(\theta)+\int_{\Omega}J d\mathbf{P}_{\nu'}(\omega)\Big\}\\
&=&P(\theta,J),\notag
\end{eqnarray}
where the supremum is taken over all $\mathbf{P}_{\nu'}=\nu'^{\mathbb{N}}$ is the product measure of some Borel probability measure $\nu'$  on $\mathfrak{F}$ with $\nu'_i=\nu'(f_i)$, and in the last line we use the variational principle for the topological pressure $P(\theta,\eta_J)$ of $J$ with respect to $\theta$. From \cite[Chapter 9 ]{Walters}, we get that,
\begin{equation}\label{pressureJ}
P(\theta,J)=\log
\Big(\sum_{i=1}^2\exp(\sum_{\lambda_{j_k}(x,f_i)>0}\lambda_{j_k}(x,f_i)m_{j_k}(x,f_i))\Big).
\end{equation}
Therefore, by (\ref{pressure}) and (\ref{pressureJ}), $h(\sigma_{\mathbf{f}})\leq\log
\Big(\sum_{i=1}^2\exp(\sum_{\lambda_{j_k}(x,f_i)>0}\lambda_{j_k}(x,f_i)m_{j_k}(x,f_i))\Big).$
 Moreover, by
\cite[Theorem 9.16]{Walters}, $J$ has a unique equilibrium state which is the product measure defined by the measure on $\mathfrak{F}$ which gives the element $f_i$, $i=1,2$, measure
$$
\nu_i=\frac{\sum_{\lambda_{j_k}(x,f_i)>0}\exp(\lambda_{j_k}(x,f_i)m_{j_k}(x,f_i)}{\sum_{i=1}^2\sum_{\lambda_{j_k}(x,f_i)>0}\exp(\lambda_{j_k}(x,f_i)m_{j_k}(x,f_i))}.
$$
So any measure $\nu$ defined by above $\nu_i$ satisfies that the product measure $\mathbf{P}_{\nu}\times\mu$ is a measure with maximal entropy of $F$.

Furthermore, if $\mu$ is ergodic and $\mu (\{x\in X:f_1(x)=f_2(x)\})=0$, then we conclude that $\tilde{\pi}$ is one-to-one on a set of full $\mathbf{P}_{\nu}\times\mu$ measure.  So
 \begin{equation}\label{equality2}
h_{\mathbf{P}_{\nu}\times\mu}(F)=h_{\tilde{\pi}(\mathbf{P}_{\nu}\times\mu)}(\sigma_{\mathbf{f}}).
 \end{equation}
Moreover, by the variational principle for $\sigma_{\mathbf{f}}$  we have that
 \begin{equation}\label{equality3}
h_{\tilde{\pi}(\mathbf{P}_{\nu}\times\mu)}(\sigma_{\mathbf{f}})\leq h(\sigma_{\mathbf{f}}).
 \end{equation}
By (\ref{pressure}), (\ref{equality2}) and (\ref{equality3}), $h(F)\leq h(\sigma_{\mathbf{f}})$. Together with the previous inequality $h(\sigma_{\mathbf{f}})\leq h(F)$ we have $h(F)=h(\sigma_{\mathbf{f}})$, and hence by (\ref{pressure}) and (\ref{pressureJ}), formula (\ref{Fried3}) holds.
\end{proof}

\hspace*{-0.15in}{\bf Acknowledgements.} The first author is supported by NSFC (No: 11871394) and Natural Science Foundation of Shaanxi Province (2020JC-39), the second author is supported by NSFC (No: 11771118). The first author would also like to thank Professor Jon Aaronson and School of Mathematical Sciences of Tel Aviv University for hospitality during his visit there.

\end{document}